\documentclass[12pt,twoside]{article}
\usepackage[all]{}
\usepackage {amssymb,latexsym,amsthm,amscd,amsmath}
\usepackage{times}
\usepackage{mathptm}

\newtheorem{theorem}{Theorem}[section]
\newtheorem{lemma}{Lemma}[section]
\newtheorem{corollary}{Corollary}[section]
\newtheorem{proposition}{Proposition}[section]
\newtheorem*{definition}{Definition}
\newtheorem*{theorem*}{Theorem}
\theoremstyle{definition}

\newtheorem{remark}[theorem]{Remark}
\numberwithin{equation}{subsection}

\newcommand{\ignore}[1]{}

\newcommand{\mynote}[1]{}
\begin{document}
\setcounter{section}{0}
% document information
\title{\bf Embeddings of Rank- $2$ tori in Algebraic groups}
%\author{Maneesh Thakur}
\author{Neha Hooda\\ \small {Indian Statistical Institute, 7-S.J.S. Sansanwal Marg} \\ \small {New Delhi 110016, India} \\ \small {e-mail: neha.hooda2@
gmail.com}}
%\keywords{Automorphisms, Albert algebras, Structure group, inner structure group, Kneser-Tits}
\date{}
\maketitle
\begin{abstract}
\noindent
Let $k$ be a field of characteristic different from $2$ and $3$. In this paper we study connected simple algebraic groups of type $A_2$, $G_2$ and $F_4$ defined over $k$, via their rank-$2$ $k$-tori. 
Simple, simply connected groups of type $A_2$ play a pivotal role in the study of exceptional groups and this aspect is brought out by the results in this paper. We refer to tori, which are maximal tori of $A_n$ type groups, as unitary tori. We discuss conditions necessary for a rank-$2$ unitary $k$-torus to embed in simple $k$-groups of type $A_2$, $G_2$ and $F_4$ in terms of the mod-$2$ Galois cohomological invariants attached with these groups. We calculate the number of rank-$2$ $k$-unitary tori generating these algebraic groups (in fact exhibit such tori explicitly). The results in this paper and our earlier paper (\cite{NH}) show that the mod-$2$ invariants of groups of type $G_2,F_4$ and $A_2$ are controlled by their $k$-subgroups of type $A_1$ and $A_2$ as well as the unitary $k$-tori embedded in them. 
\end{abstract}
%%%%%%%%%%%%%%%%%%%%%%%%%%%%%%%%%%%%%%%%%%%
\section{\bf Introduction}
The main aim of this paper is to investigate embeddings of rank-$2$ tori in groups of type $G_2,F_4$ and $A_2$. In our earlier work (\cite{NH}), we studied $k$-embeddings of connected, simple algebraic groups of type $A_1$ and $A_2$ in simple groups of type $G_2$ and $F_4$, defined over $k$, in terms of their respective mod-$2$ Galois cohomological invariants. We also showed that these groups are generated by their $A_2$ type $k$-subgroups. Owing to these results, importance of groups of type $A_2$ becomes evident in studying exceptional groups.

In the present paper, the mod-$2$ invariants of groups of type $F_4,G_2$ and $A_2$ are studied via the embeddings of certain rank-$2$ $k$-tori. To a simple, simply connected algebraic group $G$ of type $F_4$, $G_2$ or $A_2$ defined over $k$, one attaches certain mod-$2$ Galois cohomological invariants, which are the Arason invariants of 
some Pfister forms attached to 
these groups. Let $G$ be a group of type $F_4$ defined over $k.$ Then there exists an Albert algebra $A$ over $k$ such that $G={\bf Aut}(A)$, the full group of automorphisms of $A$. To any Albert algebra $A$, one attaches a certain reduced Albert algebra $\mathcal{H}_3(C,\Gamma)$, for an octonion algebra $C$ over $k$ and $\Gamma = Diag( \gamma_1,\gamma_2,\gamma_3) \in GL_3(k)$ (\cite{PR}). This defines two mod-$2$ invariants for 
$G={\bf Aut}(A)$:
$$f_3(G)=f_3(A):=e_3(n_C)\in H^3(k,\mathbb{Z}/2\mathbb{Z}),$$
$$f_5(G)=f_5(A):= e_5(n_C \otimes <<-\gamma_1^{-1}\gamma_2, -\gamma_2^{-1}\gamma_3>>)\in H^5(k,\mathbb{Z}/2\mathbb{Z}),$$ where $e_3$ and $e_5$ are the Arason invariants of the respective Pfister forms and $n_C$ is the norm form of $C$. 
We set $Oct(G)=Oct(A):=C$. 
Similarly, an algebraic group of type $G_2$ defined over $k$ is precisely of the form 
${\bf Aut}(C)$, for an octonion algebra $C$ over $k$ with norm form $n_C$, and this is classified by the Arason invariant $f_3(G)=e_3(n_C)\in H^3(k,\mathbb{Z}/2\mathbb{Z})$. Define $Oct(G):= C$. Finally, let $G$ be a simple, simply connected group of type $A_2$ defined over $k$. To such a group $G$, one attaches an invariant $f_3(G)\in H^3(k,\mathbb{Z}/2\mathbb{Z})$, the Arason invariant of a $3$-fold Pfister form over $k$, which is the norm form of an octonion algebra $C$. We define $Oct(G):= C$ (see \S 2.5).

Let $L, K$ be \'etale algebras over $k$ of dimensions $3, 2$ resp. and $T= {\bf SU}(L\otimes K, 1 \otimes ~\bar{}~),$ where ~$\bar{}$~ denotes the non-trivial involution on $K$. Then $T$ is a torus defined over $k$, referred to in the paper as the $K$-{\bf unitary torus} associated with the pair $(L, K)$. For this torus, we let $q_T:= <1, -\alpha \delta>=N_{k(\sqrt{\alpha\delta})/k}$, where $Disc(L)= k(\sqrt{\delta})$ and $K= k(\sqrt {\alpha})$. Such tori are important as they occur as maximal tori in simple, simply connected groups of type $A_2$ and $G_2$. We will be interested in conditions under which such tori embed in groups of type $A_2, G_2$ or $F_4$ defined over $k$. A unitary torus $T$ will be called {\bf distinguished} over $k$ if $q_T$ is hyperbolic over $k$, or equivalently, if $Disc(L)=K$. We shall see that the behaviour of the invariant $f_3$ for groups of type $A_2$ and $G_2$ is somewhat analogous to the behaviour of the invariant $f_5$ for groups of type $F_4$, as far as embeddings of unitary tori in such groups is our concern. 

Now we describe our main results. Let $L, K$ be \'etale algebras over $k$ of dimensions $3, 2$ resp. and $T$ be the $K$-{\bf unitary torus} associated with the pair $(L, K)$. Let $G$ be a simple, simply connected $k$-group of type $A_2$ or $G_2$. We prove that if $T \hookrightarrow G$ over $k$, then $q_T$ divides $f_3(G)$ (see Theorem \ref{fact1}). We show that $G$ contains a distinguished maximal $k$-torus if and only if $f_3(G)= 0$ (see  Theorems \ref{g2to}, \ref{dist}). Similarly, for the groups of type $F_4$ we prove the following:  Let $G$ be a simple, simply connected $k$-group of type $F_4$. If $T \hookrightarrow G$ over $k$, then $q_T$ divides $f_5(G)$ (see Theorems \ref{albi}). We show that a $k$-group $G$ of type $F_4$ contains a distinguished $k$-torus if and only if $f_5(G) = 0$ (see Theorems \ref{f5}). Let $G$ be a $k$-group of type $G_2$ or a simply connected, simple group of type $A_2$ defined over $k$. We prove that a $k$-embedding $T \hookrightarrow G$ forces $K \subseteq Oct(G)$ (i.e, the Pfister form $<1, -\alpha>$ divides $f_3(G)$) (see Theorem \ref{KC}). This fails to hold for groups of type $F_4$, but holds in the special case when $L$ has trivial discriminant (see Theorem \ref{KC1}). Let $G$ be a $k$-group of type $F_4$. Then the existence of a $k$-embedding $T \hookrightarrow G$ for a $K=k(\sqrt{\alpha})$-unitary torus $T$ implies that the  Pfister form $<1,-\alpha>$ divides $f_5(G)$ (see Theorem \ref{albi1}).

Let $L, K$ be \'etale algebras over $k$ of dimensions $3, 2$ resp. and let $(E,\tau)=(L\otimes K,1\otimes~\bar{}~)$, where $x\mapsto \overline{x}$ is the non-trivial $k$-automorphism of $K$. We define \'etale Tits process algebras $J_1$ and $J_2$ arising from the pair $(L,K)$ to be $L$-isomorphic, if there exists a $k$-isomorphism $J_1\rightarrow J_2$ which restricts to the subalgebra 
$L$ of $J_1$ and $J_2$. We establish a relation between $H^1(k,{\bf SU}(E,\tau))$ and the set of $L$-isomorphism classes of \'etale Tits process algebras arising from $(L,K)$ (see Theorem \ref{isomclas}). We also show that $H^1(k,{\bf SU}(E,\tau))=0$ if and only if all \'etale Tits process algebras arising from $(L,K)$ are $L$-isomorphic to the \'etale Tits process $J(E,\tau, 1,1)$ (see Theorem \ref{titsisom}).  

We study next the effect of the presence of a unitary torus $T$ as above in groups of type $A_2,G_2$ and $F_4$ when $H^1(k,T)=0$. We show that a $k$-group $G$ of type $G_2$ contains a maximal $k$-torus 
$T\subset G$ such that $H^1(k,T)=0$ if and only if the associated mod-$2$ invariant $f_3(G)$ vanishes, i.e, $G$ splits (see Theorem \ref{introg2}). We derive a similar result for simple, simply connected $k$-groups of type $A_2$: if such a group $G$ has a maximal $k$-torus $T$ with $H^1(k,T)=0$, then $f_3(G)=Oct(G)$ splits.  The converse holds in the case when the group arises from a matrix algebra (see Theorem \ref{cohomoo}). For $k$-groups of type $F_4$, we can prove a weaker result. Let $G$ be a group of type $F_4$ (resp. $G_2$ or a simple, simply connected group of type $A_2$) defined over a field $k$. Assume that there is a $k$-embedding $T \hookrightarrow G$ for a unitary torus $T$ associated to an \'etale pair $(L,K)$ as above. If $H^1({\bf U}(L \otimes K, ~1 \otimes~\bar{}~))= 0$ then $f_5(A)= 0$ (resp. $Oct(G)$ splits) (see Theorem \ref{introf4}).

In the final section of this paper, we compute the number of rank-$2$ $k$-tori ($k$ a perfect field) generating simple, simply connected $k$-groups of type $A_2$, $G_2$, and $F_4$, arising from division algebras and for $k$-subgroups of type $D_4$ of ${\bf Aut}(A)$, where $A$ is an Albert division algebra. In fact, we explicitly exhibit such $k$-tori in each case.
It seems likely that these numbers are minimal in each case. The numbers are mentioned in the table below. 
\begin{table}[ht]
\caption{Number of $k$-tori required for generation of groups} 
\centering % used for centering table
\begin{tabular}{c c } % centered columns (4 columns)
\hline\hline %inserts double horizontal lines
Type of group & Number of rank-$2$ $k$-tori required for generation   \\ [1ex] % inserts tabl
%heading
\hline % inserts single horizontal line
$A_2$ & 2  \\ 
$G_2$ & 3  \\
$D_4$ & 3  \\
$F_4$ & 4   \\ [1ex]
\hline %inserts single line
\end{tabular}
\label{table:nonlin} % is used to refer this table in the text
\end{table}
\vskip1mm
\noindent
At the time of submission of this paper, we discovered the paper (\cite{CPT}) by C. Beli, P. Gille and T.-Y, Lee, posted recently on the math arXiv.  They have studied maximal tori in groups of type $G_2$. Some of our results on groups of type $G_2$ partially match with results in this paper (see \cite{CPT}, Proposition 4.3.1., Corollary 4.4.2., Remarks 5.2.5. (b), Proposition 5.2.6 (i)), however the scope of our paper and methods of proofs are different. 
\vskip2mm
\noindent
{\bf Structure of the paper:} In this paper, some of the results are valid over fields of general characteristics, however, for simplicity, we will work over a base field $k$ of characteristic different from $2$ and $3$. We now proceed to describe the structure of the paper briefly: 

\S 2 contains material, mainly preliminary in nature, on Albert algebras, octonion algebras, algebraic groups and the theory of unitary involutions on central simple algebras. We also fix some notation and terminology to be used in the rest of the paper. In \S 3, \S 4, we study the $k$-embeddings of rank-$2$ unitary $k$-tori in algebraic groups of type $A_2,~G_2$ and $F_4$ defined over $k$, in terms of the mod-$2$ Galois cohomological invariants attached to these groups. In \S 4.1, we discuss cohomology computations of unitary tori. \S 5 contains some applications of the cohomology computations in \S 4.1 to \'etale Tits process algebras. In \S 5.1 we mainly discuss consequences of presence of maximal $k$-tori $T$ with $H^1(k,T)=0$, in $k$-groups of type $A_2$ and $G_2$, on the mod-$2$ invariants of such groups. In the final section \S 6, we prove the generation of $k$-groups of type $A_2,~G_2$ and $F_4$ (arising from division algebras) by their rank-$2$ unitary $k$-tori. We also compute the numbers of such tori, and exhibit such tori explicitly, in each of these cases.       
\section{Preliminaries}
\subsection{\bf Notations} In this section we collect some notations which will be used in the paper. Let $A$ be an associative algebra. By $A^*$ we shall denote the group of units in $A$. In particular, for a field $k$, $k^*:= k -\{0\}$. Let $L$ be a finite \'etale extension of $k$. 
We will denote the norm and the trace maps by $N_{L/k}$, $T_{L/k}$ respectively. By $R_{L/k}^{(1)}(\mathbb{G}_m)$ we shall denote the $k$-torus 
$Ker\{ N_{L/k}: L^* \rightarrow k^*\}$ of norm $1$ elements in $L$ and, when convenient, this torus will also be denoted by ${\bf L}^{(1)}$. Let $G$ be an algebraic group defined over $k$. By $G(k)$ we denote the group of $k$-rational points in $G$. For a finite \'etale extension $L$ of $k$ the group of $k$-points of ${\bf L}^{(1)}$ will be denoted by $L^{(1)}$. We define an algebraic group $G$ to be simple if $G$ has no non-trivial proper connected normal subgroups.
By $<<a_1, a_2,\cdots,a_n>>$ we shall mean the $n$-fold Pfister form $ <1, -a_1> \otimes  <1, -a_2> \otimes\cdots\otimes  <1, -a_n>$ over $k$.
\noindent
In the paper, unadorned tensor products will be understood to be over base fields and dimensions, when not specified, are over the base fields. For an object $X$ defined over $k$, $X \otimes_k L$ will denote the base change $X \times_k L$ of $X$. By $K=k(\sqrt{\alpha})$ we denote the quadratic algebra $k[x]/(x^2-\alpha)$ for $\alpha \in k^*$.
\subsection{\bf Unitary Involutions and their invariants}
We need some results form the theory of unitary involutions on central simple algebras, for details we refer to (\cite{HKRT}, \cite{KMRT}). Let $k$ be a field with characteristic different from $2$. Let $K$ be a quadratic \'etale extension of $k$ and let $B$ be a central simple algebra of degree $3$ over $K$ with an involution $\sigma$ of the second kind. Let $T_B$ be the reduced trace map on $B$ and $Q_{\sigma}$ be the restriction of the trace form $Q: (x, y) \rightarrow T_B(xy)$ to $(B, \sigma)_+$, the $k$-space of symmetric elements of $B$. Then the decomposition of  $Q_{\sigma}$ is given by 
\begin{proposition}   ( \cite{HKRT}, \S 4) Let $K=k(\sqrt{\alpha})$. Then there exist $b,c \in k^*$ such that,
$$Q_{\sigma}\cong  <1,1,1> \bot  <2>.<<\alpha>>.<-b,-c,bc>.$$
\end{proposition}
\noindent
In (\cite{HKRT}, Theorem 15), it was proved that the $3$-fold Pfsiter form $f_3(B, \sigma):= <<\alpha, b, c>>$ is an invariant for $\sigma$. The unitary involution $\sigma$ is called {\bf distinguished} if $f_3(B, \sigma)$ is hyperbolic (see \cite{HKRT}, \S 4). For $B= M_3(K)$, up to automorphisms of $(B, \sigma)$, we
have $ \sigma = Int(a) \circ \tau$, where $\tau(x_{ij})=(\overline{x_{ij}})^t$ with $a= diag(a_1,a_2,a_3)\in GL_3(k)$. By  (\cite{HKRT}, Prop.  2) we have,
 $$Q_{\sigma} \cong  <1,1,1> \bot <2>.<<\alpha>>.<a_1a_2,a_1a_3,a_2a_3>.$$
In this case  $f_3(B, \sigma)= <<\alpha, -a_1a_2, -a_2a_3>>$.
 Hence, if $\sigma$ is distinguished and $K$ is a field, then $<a_1a_2,a_1a_3,a_2a_3>_K \cong <1,-1,-1>_K$. \\
\vskip0.5mm
\noindent
The Arason invariant $e_n(<<a_1,a_2,\cdots,a_n>>)$ of the $n$-fold Pfister form $<< a_1,a_2,\cdots,a_n>>$ is given by, 
$$e_n(<< a_1,a_2,..,a_n>>)=(a_1)\cup(a_2)\cup\cdots\cup(a_n) \in H^n(k,\mathbb{Z}/2\mathbb{Z}),$$
where, for $a\in k^*,~(a)$ denotes the class of $a$ in $H^1(k,\mathbb{Z}/2\mathbb{Z})$ (see \cite{A}, Pg. 453).
\begin{proposition}( \cite{HKRT}, Prop. 17) Let $B$ be as above. For every cubic \'etale $k$-subalgebra $L \subseteq B$, there is a distinguished involution $\sigma$ such that $L \subseteq (B, \sigma)_+$.
\end{proposition}
\subsection{\bf Albert Algebras}
{\bf Octonion algebras and Pfister forms:}
Let $C$ be an octonion algebra over $k$ and let $n_C$ denote its norm form. Then $C$ is determined, up to 
isomorphism by $n_C$, which is a $3$-fold Pfister form over $k$. Conversely, any $3$-fold Pfister form is the norm form of a unique (up to isomorphism) octonion algebra over $k$. Recall that an octonion algebra $C$ over $k$ is {\bf split} if and only if the associated norm form $n_C$ is isotropic over $k$. Note that a Pfister form $q$ over $k$ is either hyperbolic over $k$ or is $k$-anisotropic (\cite{TYL}, Chap. X, Theorem 1.7). Let $C$ be an octonion algebra over $k$ and $K\subseteq C$ be a quadratic \'etale subalgebra. Then $K^{\perp}$ in $C$ with respect to the norm form on $C$, has a rank-$3$ hermitian module structure over $K$. We record this below:  
\begin{proposition}{\label{hermitian}}(\cite{NJ1}, \S 5) Let $C$ be an octonion algebra over $k$ and $K\subseteq C$ be a quadratic \'etale subalgebra. Then $K^{\perp}\subseteq C$ has a rank-$3$ $K$-hermitian module structure given as follows:\\
Let $K=k(\sqrt{\alpha}),~\alpha\in k^*$. Define $h:K^{\perp}\times K^{\perp}\longrightarrow K$ by 
$$h(x,y)=N(x,y)+\alpha^{-1}N(\alpha x,y),$$
where $N(x,y)$ is the norm bilinear form of $C$ and $K$ acts on $K^{\perp}$ from the left via the multiplication in $C$. 
\end{proposition}
\noindent
Let $C$ be an octonion algebra over $k$ with $x\mapsto \bar{x}$ as its canonical involution. Let $M_3(C)$ denote the algebra of $3 \times 3$ matrices with entries in $C$. 
Let $\Gamma= diag(\gamma_1, \gamma_2, \gamma_3)\in GL_3(k)$. Let $$\mathcal{H}_3(C, \Gamma) = \{ X \in M_3(C)| \Gamma^{-1}\overline{X}^{t}\Gamma = X\},$$
where, for $X = (x_{ij}), \overline{X}:=  (\overline{x_{ij}})$, and $X^t$ is the transpose of $X$. The product  $X \circ Y := \frac{1}{2}(XY+YX)$ defines a simple Jordan algebra structure on $\mathcal{H}_3(C, \Gamma)$, here $X, Y \in \mathcal{H}_3(C, \Gamma)$ and $XY$ is the product of $X$ and $Y$ in $M_3(C)$. With this structure, $\mathcal{H}_3(C, \Gamma)$, is called a {\bf reduced Albert algebra}. 
\begin{definition} A $k$-algebra $A$ is called an Albert algebra if, over some field extension $L$ of $k$, $A \otimes_k L$ is isomorphic to a reduced Albert algebra.
\end{definition}
\vskip0.5mm
\noindent
{\bf Reduced models for Albert algebras:}\\
\noindent
Let $A$ be an Albert algebra over $k$. Then there exists, up to a $k$-isomorphism, a unique reduced Albert algebra $A_{red}:=\mathcal{H}_3(C,\Gamma)$ over $k$, such that, for any extension $L$ of $k$ with $A \otimes_k L$ reduced, $A \otimes_k L \cong\mathcal{H}_3(C \otimes_k L, \Gamma)$ (see \cite{PR7}, \S 7). We call $A_{red}$ the {\bf reduced model} of $A$. By (\cite{PR7}, \S4, \S7), the isomorphism class of $A$ determines $A_{red}$, and hence $C$, up to $k$-isomorphism. If $A_{red}=\mathcal{H}_3(C,\Gamma)$, we call $C$ the {\bf octonion algebra} of $A$, denoted by ${\bf Oct}(A)$.
\vskip0.5mm
\noindent
{\bf Tits construction of Albert algebras:}
 Let $k$ be a field of characteristic different from $2$ and $3$. There are two rational constructions of Albert algebras due to Tits, which are exhaustive (though not exclusive). We briefly describe them below:
\vskip0.5mm
 \noindent
{\bf Tits's first construction}\\
\noindent
Let $D$ be a central simple algebra of degree $3$ over a field $k$ and let $N_D$ and $T_D$ denote respectively the reduced norm and reduced trace maps on $D$. Let $\mu \in k^*$. Let $J(D, \mu):= D_0 \oplus D_1 \oplus D_2$, where  $ D_i =D$ for $ i= 1,2,3$. Then there is a product on  $J(D, \mu)$ with $(1, 0, 0)$ as identity and $$N(x, y, z)= N_D(x)+ \mu N_D(y)+ \mu^{-1}N_D(z) -T_D(xyz),~ x,~ y,~ z \in D,$$ as cubic norm, such that  $J(D, \mu)$ is an Albert algebra over $k$, referred to as a {\bf first Tits construction} Albert algebra. It is known that $J(D, \mu)$ is a division algebra if and only if $N(x, y, z)$ has no non-trivial zeros, if and only if $\mu$ is not a reduced norm from $D$ (\cite{KMRT}, Remark 39.15). By (\cite{KMRT}, Prop. 40.5), an Albert algebra $A$ over $k$ is a first construction if and only if $Oct(A)$ is split over $k$.
\vskip0.5mm
\noindent
{\bf Tits's second construction}\\
\noindent
Let $K$ be a quadratic \'etale extension of $k$ and $B$ be a central simple algebra of degree $3$ over $K$. Let $\sigma$ be an involution of the second kind on $B$. Let $ x \mapsto\overline{x}$ denote the non-trivial Galois automorphism of $K/k$.
Let $(B, \sigma)_+$ be the $k$-subspace of $B$ of $\sigma$-symmetric elements in $B$. Fix a unit $u$ in $(B, \sigma)_+$ such that $N_B(u)=\mu\overline{\mu}$ for some $\mu\in K^*$. Let $J(B, \sigma,u, \mu):= (B, \sigma)_+ \oplus B$. There is a product on $J(B, \sigma,u, \mu)$ with $(1,0)$ as identity and $$N(a, x)= N_B(a) +\mu N_B(x)+ \overline{\mu}N_B(\sigma(a))- T_B(axu\sigma(a)),~ a \in (B, \sigma)_+,~ x \in B,$$ as cubic norm, such that $J(B, \sigma,u, \mu)$ is an Albert algebra over $k$, referred to as a {\bf second Tits construction} Albert algebra.
One knows that $J(B, \sigma,u, \mu)$ is a division algebra if and only if $N(a, x)$ has no non-trivial zeros, if and only if
$\mu$ is not a reduced norm from $B$ (\cite{KMRT}, Theorem 39.18). 
\vskip0.5mm
\noindent 
\begin{theorem}(\cite{J}, Chapter. VI, Thm. 7){\label{jordoisomp}} Let $k$ be a field of characteristic different from $2$ and $3$.. Let $J$, $J'$ be finite dimensional Jordan algebras over $k$ with identity such that $J'$ is separable. If $\eta$ is a norm isometry of $J$ onto $J'$ preserving identities, then $\eta$ is an isomorphism of $k$-algebras.
\end{theorem}
\subsection{Algebraic groups}
In this section, we collect some results on algebraic groups that we will need. We begin with some notation and terminology. For a finite dimensional $k$-algebra $A$, ${\bf Aut}(A)$ will denote the algebraic group $ Aut_{\overline{k}}(A\otimes_k\overline{k})$ of algebra automorphisms defined over $k$. For a connected algebraic group $G$ defined over $k$, the {\bf $k$-rank} of $G$ is defined as the dimension of a maximal $k$-split torus contained in $G$. \\
\noindent
For a connected reductive group $G$ over $k$, we say $G$ is $k$-{\bf isotropic} if there exists a non-central $k$-split torus in $G$ and $k$-{\bf anisotropic} otherwise. By the {\bf type} of a connected reductive algebraic group $G$ defined over $k$, we mean the Cartan-Killing type of its root system. By the {\bf absolute rank} of an algebraic group $G$ defined over $k$ we mean the dimension of a maximal torus in $G$. By the {\bf rank} of a torus $T$ we mean its dimension. \\ 
\noindent
Let $A$ be a finite dimensional $k$-algebra and $S\subset A$ be a $k$-subalgebra. In the rest of the paper, we shall denote by ${\bf Aut}(A/S)$ the $k$-subgroup of ${\bf Aut}(A)$ consisting of all automorphisms of $A$ which fix $S$ pointwise and ${\bf Aut}(A, S)$ will denote the $k$-subgroup of ${\bf Aut}(A)$ mapping $S$ to $S$.
\vskip0.5mm
\noindent 
{\bf Unitary groups and Unitary tori :} Let $K$ be a quadratic \'etale extension of a field $k$ and let $B$ be either a central simple $K$-algebra or an \'etale $K$-algebra in the sense of (\cite{KMRT}, \S 18.A). Assume that there is an involution $\sigma$ on $B$ of the second kind over $K$, i.e. $\sigma$ restricts to $K$ as the non-trivial $k$-automorphism of $K$. Let $N_B$ denote the reduced norm map of the central simple algebra $B$ or the norm map on the \'etale algebra $B$. We then define the algebraic groups 
${\bf U}(B,\sigma)$ and ${\bf SU}(B,\sigma)$, by specifying the group of $L$-rational points, for any finite dimensional commutative $k$-algebra $L$, as follows :
$${\bf U}(B,\sigma)(L)=\{x\in B\otimes L|~x\sigma(x)=1\},~~{\bf SU}(B,\sigma)(L)=\{x\in {\bf U}(B,\sigma)|~N_B(x)=1\}.$$  
\noindent
We note here that ${\bf SU}(B,\sigma)$ is a simple, simply connected algebraic groups of type $A_2$ defined over $k$ when $B$ is a central simple $K$-algebra of degree $3$. When $B$ is an \'etale algebra over $K$ of rank $n$, then ${\bf U}(B,\sigma)$ and ${\bf SU}(B,\sigma)$ are tori defined over $k$ (of rank resp. $n$ and $n-1$).    
\noindent
We will denote ${\bf U}(B,\sigma)(k)$ by $U(B,\sigma)$ and ${\bf SU}(B,\sigma)(k)$ by $SU(B,\sigma)$. 
\begin{proposition}{\label{unipotred}}(\cite{R}, Prop. 6.3) Let $G$ be a connected reductive algebraic group defined over a perfect (infinite) field $k$, then $G$ is $k$-anisotropic if and only if $G(k)$ contains no non-trivial unipotents . 
\end{proposition}
\begin{proposition}{\label{red}} (\cite{SV}, Thm. 7.2.1, \cite{PR2}, Pg. 205) Let $A$ be an Albert algebra over a field $k$. Let $G= {\bf Aut}(A)$ be the algebraic group of automorphisms of $A$. Then $G$ is a connected simple algebraic group of type $F_4$ defined over $k$. Moreover, $G$ is $k$-isotropic if and only if $A$ is reduced and $f_5(A)=0$.
\end{proposition}
\noindent
Similarly, for groups of type $G_2$, we have the following,
\begin{proposition} ( \cite{SV}, \S 2.3, Thm. 2.3.5, \cite{Spr}, Prop. 17.4.2, 17.4.5){\label{red1}} {\bf (a)} Let $C$ be an octonion algebra over a field $k$ and let $G={\bf Aut}(C)$ be the algebraic group of automorphisms of $C$. Then $G$ is a connected, simple algebraic group of type $G_2$ defined over $k$ and $G$ is either $k$-anisotropic or $k$-split. Moreover, \\
{\bf (b)} $G$ is $k$-anisotropic if and only if $C$ is a division algebra, if and only if the norm form of $C$ is $k$-anisotropic.  
\end{proposition}
\noindent
The results below describe certain $k$-subgroups in $k$-groups of type $G_2$ and $F_4$, these will be needed in the sequel.
\begin{proposition}{\label{D4}}(Jacobson) Let $A$ be an Albert algebra defined over $k$ and let $G={\bf Aut}(A)$ denote the algebraic group of type $F_4$ associated with $A$. Let $L\subset A$ be a cubic \'etale subalgebra. Then the subgroup ${\bf Aut}(A/L)$ is a simply connected, simple group of type $D_4$ defined over $k$. 
\end{proposition}
\begin{theorem}{\label{typeA}}  (\cite{NJ4}, Thm. 9, \cite{KMRT}, \S 39, Chap. IX) Let $A$ be an Albert algebra over $k$ and let $S$ be a $9$-dimensional cubic separable Jordan subalgebra of $A$. The subgroup ${\bf Aut}(A/S)$ (resp. ${\bf Aut}(A, S)$) is a simply connected, simple algebraic group of type $A_2$ (resp. $A_2 \times A_2$) defined over $k$. 
\end{theorem}

\begin{theorem}{\label{typeA2}} (\cite{NJ1}, Thm. 3, Thm. 4, Thm. 5) Let $C$ be an octonion algebra over $k$ and let $S$ be a quadratic \'etale (resp. quaternion) subalgebra of $C$. Then the subgroup ${\bf Aut}(C/S)$ is a simply connected, simple group of type $A_2$ (resp. $A_1$) defined over $k$. 
\end{theorem}
\noindent
{\bf Group actions :} Let $G$ be a group acting on a set $X$. Let $g\in G$ and $H$ be a subgroup of $G$. We denote the set of fixed points of $g$ in $X$ by $X^g$ and the set of fixed points of $H$ in $X$ by $X^H$. We will often use the following obvious fact 
\begin{lemma} Let $g\in G$ and $h\in Z_G(g)$, the centralizer of $g$ in $G$. Then $h$ stabilizes $X^g$, i.e., maps $X^g$ to itself. Consequently, for a subgroup $H$ of $G$, its centralizer $Z_G(H)$ maps $X^H$ to itself.
\end{lemma} 
\noindent
{\bf Embeddings of $A_2$ in $F_4$ :} Let $B$ be a degree $3$ central simple algebra over a quadratic \'etale extension $K$ of $k$ with an involution $\sigma$ of the second kind.
Let $A= J(B, \sigma, u, \mu)$ be a second Tits construction Albert algebra. Let $G= {\bf Aut}(A)$. Then $G$ is a group of type $F_4$ over $k$. We have the following embedding of the special unitary group ${\bf SU}(B,\sigma)$ in $G$,\\
$${\bf SU}(B, \sigma)\hookrightarrow G~\text{via}~p\mapsto\phi_p,~\text{where}~\phi_p : (x , y) \mapsto (px \tau(p), py),~
\text{for~all}~(x,y)\in A.$$
\noindent
For a central simple algebra $D$ over $k$, ${\bf SL}_1(D)$ denotes the algebraic group of norm $1$ elements in $D$. When $A=J(D, \mu)$ is a first Tits construction Albert algebra, we have the embedding of ${\bf SL}_1(D)$ in $G= {\bf Aut}(A)$ given by,
$${\bf SL}_1(D) \hookrightarrow G~\text{via}~p\mapsto \phi_p,~\text{where}~\phi_p : (x, y, z) \mapsto (x, yp, p^{-1}z),~\text{for~all}~(x, y, z) \in A.$$
\subsection{ \bf Mod-$2$ invariants for groups of type $A_2$, $G_2$ and $F_4$}
Let $G$ be a group of type $G_2$ defined over $k$. Then $G \cong {\bf Aut}(C)$ for a suitable octonion algebra $C$ over $k$  (\cite{Spr}, \S 17.4). Recall that $C$ is determined by its {\bf norm form} $n_C$, which is a $3$-fold Pfister form over $k$. Hence the groups $G$ over $k$ of type $G_2$ are classified by the Arason invariant $f_3(G):= e_3(n_C) \in H^3(k,\mathbb{Z}/2\mathbb{Z})$, where $G \cong {\bf Aut}(C)$ as above. \\
\noindent
 Let $G$ be a group of type $F_4$ defined over $k$. Then there exists an Albert algebra over $k$ such that $G={\bf Aut}(A)$, the full group of automorphisms of $A$ (\cite{Spr}, \S 17.6). Let $A$ be an Albert algebra over $k$ and $A_{red}=\mathcal{H}_3(C,\Gamma)$ be the reduced model for $A$, as defined in \S 2.3. This defines two mod-$2$ invariants for $G={\bf Aut}(A)$:
$$f_3(G)=f_3(A):=e_3(n_C)\in H^3(k,\mathbb{Z}/2\mathbb{Z}),$$
$$f_5(G)=f_5(A)=e_5(n_C \otimes \langle 1,\gamma_1^{-1}\gamma_2\rangle \otimes\langle 1,\gamma_2^{-1}\gamma_3\rangle)\in H^5(k,\mathbb{Z}/2\mathbb{Z}).$$    
\begin{proposition}{\label{coprime3}} Let $D$ be a degree $3$ central division algebra or an Albert division algebra over $k$. Then $D$ remains a division algebra over field extensions of degree coprime to $3$.
\end{proposition}
\begin{proof} First, let $D$ be a degree $3$ central division algebra over $k$. By (\cite{NJ2}, Exercise $9$, Section 4.6), it follows that, for a field extension $L$ of $k$ of degree coprime to $3$, $D\otimes L$ is a division algebra. When $D$ is an Albert division algebra, the result follows from (\cite{PR2}, Cor., p. 205). 
\end{proof}
\noindent
{\bf Octonion algebras for groups of type $F_4$, $G_2$ and $A_2$}
\vskip1mm
\noindent
To a simple, simply connected algebraic group $G$ of type $F_4$, $G_2$ or $A_2$ defined over $k$, we associate an octonion algebra as follows. Let $G$ be a $k$-group of type $F_4$. Then $G \cong {\bf Aut}(A)$, for an Albert algebra $A$ over $k$. Let $A_{red}=\mathcal{H}_3(C,\Gamma)$ be the reduced model for $A$, where $C$ is an octonion algebra over $k$ and $\Gamma=diag(\gamma_1,\gamma_2,\gamma_3)\in GL_3(k)$. Define $Oct(G):= Oct(A)= C$. Observe that $f_3(A)=0$ if and only if $Oct(G)$ splits. Let $G$ be a $k$-group of type $G_2$. Then $G \cong {\bf Aut}(C)$, for an octonion algebra $C$ over $k$. Define $Oct(G):=C$. Observe that $f_3(G)=0$ if and only if $G$ splits, if and only if $Oct(G)$ splits. Let $G$ be a simple, simply connected $k$-group of type $A_2$, then $G \cong {\bf SU}(B, \sigma)$, for some degree $3$ central simple algebra $B$ over a quadratic \'etale extension $K$ of $k$, with an involution $\sigma$ of the second kind. Define $Oct(G):=C$, where $C$ is the octonion algebra determined by the $3$-fold Pfister form $f_3(B, \sigma)$. The Arason invariant $f_3(B, \sigma)$ of the norm form of the octonion algebra of $(B, \sigma)_+$ is an invariant for $G$ as well, we will denote this by $f_3(G)$, see (\cite{NH}, Remark 2.8). Note that $\sigma$ is distinguished if and only $f_3(B, \sigma)=0$ if and only if $Oct(G)$ splits. Moreover, these constructions are functorial and respect base change.

\begin{lemma}{\label{isotropycond}}{\bf (a)} Let $G$ be a group of type $F_4$ defined over $k$. If $G$ has k-rank $\geqq 1$, then $f_5(G)=0$. Moreover, if $G$ has k-rank $\geqq 2$, then $G$ splits over $k$ and $f_3(G)=0= f_5(G)$. \\
\noindent
{\bf (b)} Let $G$ be a simple, simply connected group of type $G_2$ or $A_2$ defined over $k$. If $k$-rank of $G \geq 1$, then $Oct(G)$ splits.
\end{lemma}
\begin{proof}
{\bf (a)} Let $G$ be a group of type $F_4$ defined over $k$. If $G$ has k-rank $\geqq$ $1$ then $G$ is k-isotropic and by (\cite{PR2}, Pg. 205), $f_5(G)=0$. If $G$ has k-rank $\geqq$ $2$ then by (\cite{T3}, Pg. 60), $G$ splits over $k$. Hence $A$ and thereby $Oct(A)$ splits over $k$ (\cite{Spr}, Chap. 17, \S 17.6.4, \cite{HH}, Pg. 164). \\
\noindent
{\bf (b)} Let $G$ be a group of type $G_2$ defined over $k$. If $k$-rank of $G \geq 1$, then by (\cite{Spr}, Chap. 17, \S 17.4.2), $G$ is $k$-split. Hence $Oct(G)$ is split (\cite{Spr}, Chap. 17, \S 17.4.5, \cite{T3}, Pg. 60). Let $G$ be a simple, simply connected group of type $A_2$ defined over $k$. If $k$-rank of $G \geq 1$ then $G$ is $k$-isotropic and by (\cite{NH}, Prop. 3.4), $Oct(G)$ splits.
\end{proof}
\noindent
{\bf Remark on notation :} In the paper, to avoid surplus of notation, we will often confuse between the mod-$2$ invariants, $f_3$, $f_5$ (resp. $f_3$) of groups of type $F_4$ (resp. $G_2$), $f_3$-invariant for simple, simply connected groups of type $A_2$ and the corresponding Pfister quadratic forms whenever no confusion is likely to arise. Let $q$ be an $n$-fold Pfister form and $e_n(q) \in H^n(k, \mathbb{Z}/2\mathbb{Z})$ be the Arason invariant of $q$. We will write $e_n(q)= 0$ to mean $q$ is hyperbolic. Ler $q_1, q_2$ be Pfister forms over $k$. We say $q_2$ divides $q_1$ over $k$ if there exists a Pfister form $q_3$ over $k$ such that $q_1 = q_2 \otimes q_3$ over $k$. If $q_2$ divides $q_1$ over $k$ then by Theorem \ref{subformfactor}, $q_2$ is a subform of $q_1$ over $k$.
\subsection{\bf \'Etale algebras}

Let $k$ be a field of characteristic different from $2, 3$ and $L$ be an \'etale $k$-algebra of dimension $n$. Let $T: L \times L \rightarrow F$ be the bilinear form induced by the trace, $T(x, y)= T_{L/k}(xy)$, for $x, y \in L$. Let $\delta(L)$ denote the discriminant algebra of $L$ over $k$.
\begin{proposition}(\cite{KMRT}, Prop. 18.24) Let  $L$ be an \'etale $k$-algebra of dimension $n$. Then $\delta(L)= k[T]/(t^2-d)$, where $d \in k^*$ represents the determinant of the bilinear form $T$.
\end{proposition}
\noindent
For the special case when $L$ is a cubic \'etale $k$-algebra, by (\cite{KMRT}, Prop. 18.25) we have,
\begin{proposition}{\label{disccubic}} Let $L$ be an \'etale algebra of dimension $3$ over $k.$ There is a canonical $k$-isomorphism $L \otimes L \cong L \times L \otimes \delta(L).$
\end{proposition}
\noindent
In this paper we will denote $\delta(L)$ by $Disc(L)$ and at times write also $Disc(L)= d$.
\subsection{\bf Maximal tori of special unitary groups}
Let $k$ be a field (of characteristic different from $2$, $3$) and $K$ a quadratic field extension of $k$ with the non-trivial $k$-automorphism $~\bar{}$ . Let $V$ be a $K$-vector space of dimension $n$. Let $h$ be a non-degenerate hermitian form on $V$. By (\cite{MT}, Theorem 5.1) and Corollary 5.2, we have the following well known explicit description of maximal tori in a special unitary group of a non-degenerate hermitian space, 
\begin{theorem}{\label{imp}} {\bf (a)} Let $k$ be a field and $K$ a quadratic field extension of $k.$ Let $V$ be a $K$-vector space of dimension $n$ with a non-degenerate hermitian form $h$. Let $T\subseteq {\bf U}(V, h)$ be a maximal $k$-torus. Then there exists an \'etale $K$-algebra $E_T$ of dimension $n$ over $K$, with an involution $\sigma_h$ restricting to the non-trivial $k$-automorphism of $K$, such that $T= {\bf U}(E_T, \sigma_h)$. \\
\noindent
{\bf (b)} Let $T \subset {\bf SU}(V, h)$ be a maximal $k$-torus. Then there exists an \'etale algebra $E_T$ over $K$ of dimension $n$, such that $T= {\bf SU}(E_T, \sigma_h)$
\end{theorem}
\noindent
Note that in Theorem \ref{imp}, $E_T= Z_{End_K(V)}(T)$ with involution $\sigma_h$, where $T$ is a maximal torus in ${\bf U}(V, h)$ or ${\bf SU}(V,h)$.\\
\noindent
By (\cite{MT}, Remark after Lemma 5.1) we have,
\begin{lemma}{\label{count}} Let $K$ be a quadratic \'etale extension of $k$. Let $E$ be an \'etale algebra of dimension $2n$ over $k$ containing $K$, equipped with an involution $\sigma$, restricting to the non-trivial $k$-automorphism of $K$. Let $L= E^{\sigma}= \{x \in E| \sigma(x)= x\}.$ Then $E= L {\otimes}_k K$ and $(E, \sigma)= (L {\otimes}_k K, 1 \otimes~\bar{}~)$, where $x \mapsto \overline{x}$ is the non-trivial $k$-automorphism of $K$. 
\end{lemma}
\noindent
In view of the above lemma, $dim(E^{\sigma})= n$ over $k$. Let $k$ be a field and $L, K$ be \'etale $k$-algebras of $k$-dimension $n, 2$ resp. and $E=  L \otimes K$. Then $E$ is an \'etale algebra of dimension $2n$ over $k$. Let ~$\bar{}$~ denote the non-trivial $k$-automorphism of $K$ and $\tau$  the involution $1\otimes \bar{}$ on $E$. We will refer to $(E, \tau)$ as the $K$-unitary algebra associated with the ordered pair $(L, K)$.
\begin{lemma}{\label{x}} Let $L, K$ be \'etale algebras of $k$-dimensions $n, 2$ resp. Let $(E, \tau)$ be the $K$-unitary algebra associated with $(L, K)$. Then $ E^{\tau}= \{x \in E~| \tau(x)=x\}=L$.
\end{lemma}
\begin{proof} Let $K$ be a quadratic field extension. By Lemma \ref{count}, $dim(E^{\tau})= n= dim(L)$. Since $L \subseteq E^{\tau}$, and the dimensions are equal, we have $L= E^{\tau}$. Let $K= k \times k$. Then $(E, \tau)= (L \times L, \epsilon)$, where is the switch involution on $L \times L$. Clearly $E^{\tau}= L$.
\end{proof}
\noindent
Let $L, K$ be \'etale algebras of $k$-dimensions $n, 2$ resp. and $(E, \tau)$ be the $K$-unitary algebra associated with the pair $(L, K)$. We call the torus ${\bf SU}(E, \tau)$ as the $K$-unitary torus associated to the ordered pair $(L, K)$. With such a $K$-unitary torus $T$, we associate the quadratic form $q_T:= <1, -\alpha \delta>$, where $Disc(L)= k(\sqrt{\delta})$ and $K= k(\sqrt {\alpha})$.
 Such tori are important as they occur as maximal tori in simple, simply connected groups of type $A_{n-1}$ and $G_2$. \\
\noindent
We record below an evident, yet useful result:
\begin{lemma}{\label{maxtorusina2}}   Let $K$ be a quadratic \'etale algebra over $k$ and $B$ be a degree $3$ central simple algebra over $K$ with an involution $\sigma$ of the second kind. Let $L \subseteq (B, \sigma)_+$ be a cubic \'etale subalgebra. Let $T$ be the $K$-unitary torus associated with the pair $(L, K)$. Then there exists a $k$-embedding $T \hookrightarrow {\bf SU}(B, \sigma)$.
\end{lemma}
\subsection{Maximal tori in groups of type $G_2$}
The following result is well known (cf. for example, \cite{AS1}), we supply a proof for convenience of the reader. 
\begin{proposition}{\label{maximaltorusG2}} Let $G$ be a group of type $G_2$ over $k$ and $T$ be a maximal $k$-torus of $G$. Then there exists \'etale algebras $L, K$ of $k$-dimensions $3,2$ resp. such that $T$ is the  $K$- unitary torus associated to the pair $(L, K)$.
\end{proposition}
\begin{proof} Let $G$ be as in hypothesis. Then there exists an octonion algebra $C$ over $k$ such that $G= {\bf Aut}(C).$
Let $T\subset G$ be a maximal $k$-torus in $G$.\\
\noindent
{\bf Claim:} There exists a quadratic \'etale subalgebra $K$ of $C$ such that $K= C^T$, the fixed points of the octonion algebra $C$ under the action of $T$.
\vskip0.5mm
\noindent
To see this, we may assume that the dimension $[C^t : k]= 4$ for all $t \in T(k)$. 
If not, then there exists $t \in T(k)$ such that $C^t= K$ is a quadratic \'etale subalgebra of $C$ (\cite{Wob}, cf. also \cite{HUL}, Paragraph before Theorem 4 and \cite{Alf}). Now $T$ stabilizes, and hence, by a connectedness argument, fixes $K$ pointwise. 
The claim then follows. Let $t \in T(k)$ be such that $C^t= Q$ for some quaternion subalgebra $Q$ of $C$. Since $T$ centralizes $t$, we see that $T \subseteq{\bf Aut}(C, Q)$. We write, by doubling process, $C= Q \oplus Qb$ for some $b \in Q^{\perp}$. 
Then by (\cite{SV}, \S 2.1),
$$Aut(C, Q)= \{\phi_{c, p}, c \in Q^*, p \in SL_1(Q)| \phi_{c, p}(x+ yb)\rightarrow cxc^{-1}+ (pcyc^{-1})b, \forall x, y \in Q\}.$$ 
\noindent
 By an easy computation it follows that for $c, c' \in Q^*,~ p, p' \in SL_1(Q)$ if $\phi_{c, p}\phi_{c', p'}= \phi_{c', p'}\phi_{c, p}$, then there exists $a \in k^*$ such that $cc'= ac'c$. \\
\noindent
{\bf Claim:} There exists $\phi_{c, p}\in T(k)$ such that $c \notin k^*$.
\vskip0.5mm
\noindent
If not, then for all $\phi_{c, p}\in T(k)$ we have 
$c \in k^*$. Let $x \in Q$ be arbitary and $y=0$. Then, for any $\phi_{c, p} \in T$, $\phi_{c, p}(x)= x$. 
Thus $Q \subseteq C^T$ and hence $T \subseteq{\bf Aut}(C/Q)$, where ${\bf Aut}(C/Q)$ denotes subgroup of ${\bf Aut}(C)$ consisting of automorphisms of $C$ which fixes $Q$ pointwise. 
This is a contradiction, since $T$ is a rank-$2$ torus and ${\bf Aut}(C/Q)$ is a simple group of type $A_1$.\\
\noindent
 Thus there exists $\phi_{c, p}\in T(k)$ such that  $c \notin k^*$. Since 
$\phi_{c, p}\in T$ is semisimple, $c$ generates a quadratic \'etale subalgebra, $K:= k(c)$ of $Q$. Let $\phi_{c', p'} \in T$. Since $\phi_{c', p'}$ commutes with $\phi_{c, p}$, we have $cc'= ac'c$ for some $a\in k^*$. Now, any element $\gamma\in K$ is a polynomial expression in $c$ with coefficients from $k$, say,
 $$\gamma= a_0+ a_1c+...+a_m c^m ~ \text{for} ~a_i \in k, ~m \in \mathbb{N}.$$
\noindent
Now $\phi_{c', p'}(\gamma)= c'\gamma c'^{-1}$=  $a_0+ a a_1c+\cdots+ aa^m a_mc^m \in K$. Hence $\phi_{c', p'}(\gamma) \in K$ for all $\gamma \in K$. Since $\phi_{c', p'}$ was chosen arbitrarily in $T$, we see that $T$ stabilizes, and hence, fixes $K$ pointwise.  Hence $K \subseteq C^T$. Therefore $T \subseteq {\bf SU}(K^\bot, h),$ where $h$ is the non-degenerate hermitian form on $K^\bot \subseteq C$ over $K$, induced by the norm bilinear form $n_C$ (see \cite{NJ1}, \S 5, cf. Prop. \ref{hermitian}). Note that ${\bf SU}(K^\bot, h)= {\bf SU}(M_3(K), *_h)$, where $*_h$ is the involution on $M_3(K)$ given by $*_h(X)=h^{-1}\overline{X}^th$. By Theorem \ref{imp}, any maximal torus of ${\bf SU}(M_3(K), *_h)$ is of the form ${\bf SU}(E, *_h)$ for some six dimensional $K$-unitary algebra $E$ over $k$. Hence $T= {\bf SU}(E, *_h)$ for some $E$ as above. 
\end{proof}
\noindent 
In what follows, we will call a $k$-torus $T$ as {\bf distinguished} if there exists 
\'etale $k$-algebras $L, K$ be of $k$-dimensions $3, 2$ resp. such that $disc(L)= K$ and 
$T= {\bf SU}(E, \tau)$, where $(E, \tau)$ is the $K$-unitary algebra associated to the pair 
$(L, K)$. Also observe that when $T$ is a distinguished $k$-torus, the associated quadratic form $q_T$ splits over $k$. Note also that $T$ has rank-$2$.
\section {\bf Distinguished tori in groups of type $A_2$, $G_2$ and $F_4$}
In this section we study embeddings of distinguished $k$-tori in simply connected, simple algebraic groups of type $A_2$, $G_2$ and $F_4$, defined over a field $k$, in terms of the mod-$2$ Galois cohomological invariants attached with these groups.
We prove that a group $G$ of type $F_4$ contains a distinguished $k$-torus if and only if $f_5(G)=0$. A stronger version of this result holds for groups of type $G_2$ and $A_2$. Let $G$ be a simple, simply connected group of type $G_2$ or $A_2$ defined over $k$. We prove that $Oct(G)$ splits if and only if it $G$ contains a distinguished (maximal) torus. We begin with,
\begin{theorem}{\label{isom}} Let $T$ be a distinguished torus defined over $k$. Then $T$ is isotropic over an odd degree extension of $k$.
\end{theorem}

\begin{proof} Let $T$ be a distinguished torus over $k$. Then, by definition, there exists \'etale $k$-algebras $L, K$ of $k$-dimensions $3, 2$ resp. such that $disc(L)= K$ and $T= {\bf SU}(E, \tau)$, where $(E, \tau)$ is the $K$-unitary algebra associated to $(L, K)$.
By Lemma \ref{x}, $L= E^{\tau}= \{x \in E~| \tau(x)=x\}$.
We divide the proof into three cases.
\vskip0.5mm
\noindent
{\bf  Case (i)} $L = k \times k \times k$.
\vskip0.5mm
\noindent
Since $T$ is distinguished, we have $disc(L)= K= k \times k$.
Hence $(E, \tau ) \cong (L \times L, \epsilon)$, where $\epsilon: L\times L \rightarrow  L\times L$ is given by $\epsilon(x, y)=(y, x)$, the switch involution on $L \times L.$ 
Now $$SU(E, \tau) \cong \{(x, y) \in L\times L| (x, y)\epsilon(x, y)= 1, (N_{L/k}(x), N_{L/k}(y))= (1, 1)\} \cong L^{(1)} \cong k^* \times k^*,$$ where $L^{(1)}$ denotes the group of norm $1$ elements of $L$.
It follows that ${\bf SU}(E, \tau) \cong \mathbb{G}_m \times \mathbb{G}_m$ over $k$, and hence $T= {\bf SU}(E, \tau)$ splits over $k$ in this case.
\vskip0.5mm
\noindent
{\bf  Case (ii)} $L= k \times K$, $K$ is a field.
\vskip0.5mm
\noindent
Let $\overline{\epsilon}: K \times K \rightarrow  K \times K$ be given by $\overline{\epsilon}(x, y)= (\overline{y}, \overline{x})$. Then $(E, \tau) = ((k \times K) \otimes K, 1 \otimes \bar{}~) \cong (K \times (K \times K), (~\bar{}, ~\overline{\epsilon}~))$. We have therefore,
\begin{eqnarray*}
SU(E, \tau)& \cong &\{(x, y, z)\in K \times K \times K | (x, y, z)(\overline{x}, \overline{z}, \overline{y})
= (1, 1, 1) , xyz=1\}.\\
& \cong & \{(\overline{z}z^{-1}, \overline{z}^{-1}, z)|~ z \in K^*\}= \{ (z^{-2}N(z), zN(z^{-1}), z)|~z \in K^* \}.\\
& \cong & K^*.
\end{eqnarray*}
Hence $T= {\bf SU}(E, \tau) \cong R_{K/k}(\mathbb{G}_m$) is isotropic over $k$.
\vskip0.5mm
\noindent
 {\bf  Case (iii)} $L$ is a field.
\vskip0.5mm
\noindent
 Base changing to $L$ we get, $L\otimes L \cong L \times K_0$ as $L$-algebras, where $K_0= L \otimes \Delta$, and $\Delta= Disc(L)$ (see Prop. \ref{disccubic}). By case {\bf (i)} and {\bf (ii)}, it follows that ${\bf SU}(E, \tau) \otimes L$ is isotropic. Hence $T= {\bf SU}(E, \tau)$ is isotropic over $L$.
\end{proof}
\noindent
We now study the presence of distinguished $k$-tori in groups of type $A_2$, $G_2$ and $F_4$ defined over $k$. We see that existence of such tori has a direct relation with the mod-$2$ invariants attached to these groups.
We obtain as an immediate consequence of the above theorem the following,
\begin{theorem}{\label{g2to}} Let $G$ be a group of type $G_2$ over $k$. Then $G$ splits over $k$ (equivalently, $Oct(G)$ splits over $k$) if and only if there exists a maximal $k$-torus in $G$ which is distinguished. 
\end{theorem}
\begin{proof} Let $T \subseteq G$ be a distinguished maximal $k$-torus. By Theorem \ref{isom}, $T$ becomes isotropic over an odd degree extension, say $M$, of $k$. Hence $M$-rank of $G \geq 1$. Thus $Oct(G) \otimes M$ is split (Lemma \ref{isotropycond}). By Springer's theorem, $Oct(G)$ splits over $k$ itself and consequently $G$ is $k$-split. Conversely, suppose $G$ splits over $k$. Let $L= k \times k \times k$ and $K= k \times k$ and $T= {\bf SU}(L \otimes K, 1\otimes \bar{}~)$. By case {\bf (i)} of the proof of Theorem \ref{isom}, $T \cong \mathbb{G}_m \times \mathbb{G}_m$ and $\mathbb{G}_m \times \mathbb{G}_m \hookrightarrow G$ over $k$ as $G$ is $k$-split. Hence $T$ is the required distinguished $k$-torus.
\end{proof}
\noindent
A similar result holds for groups of type $A_2$. We first recall, 
\begin{theorem}{\label{f31}} (\cite{NH}, Prop. 3.4)  Let $F$ be a quadratic \'etale $k$-algebra and $B$ be a degree 3 central simple algebra over $F$ with an involution $\sigma$ of the second kind. Then $\sigma$ is distinguished over $k$ if and only if ${\bf SU}(B,\sigma)$ becomes isotropic over an odd degree extension. 
\end{theorem}
\begin{theorem}{\label{dist}} Let $G$ be a simple, simply connected group of type $A_2$ over $k$. Then $Oct(G)$ splits over $k$ if and only if there exists a maximal $k$-torus in $G$ which is distinguished. 
\end{theorem}
\begin{proof} Let $G$ be as in the hypothesis. Then $G \cong {\bf SU}(B, \sigma)$ for some degree $3$ central simple algebra $B$ over a quadratic \'etale extension $F$ of $k$ with an involution $\sigma$ of the second kind. Let $T \hookrightarrow G$ be a maximal $k$-torus which is distinguished. Then, by Theorem \ref{isom}, $T$ is isotropic over an odd degree extension $M$ of $k$. Thus $G$ is isotropic over $M$. Hence by Theorem {\ref{f31}}, $\sigma$ is distinguished over $k$. Hence, $f_3(B, \sigma)=0$ and $Oct(G)$ is split over $k$. Conversely, if $Oct(G)$ is split over $k$, then $f_3(B, \sigma)=0$ and hence $\sigma$ is distinguished over $k$. By (\cite{HKRT}, Theorem 16, pg. 317), $(B, \sigma)_+$ contains a cubic \'etale $k$- algebra $L$ with $F$ as its discriminant algebra. Let $E= L \otimes F$. Then $E \hookrightarrow B$ and $\sigma$ restricted to $E$ equals $\tau:= 1 \otimes~\bar{}~$, where $~\bar{}~$ denotes the non-trivial $k$-automorphism of $F$. Hence ${\bf SU}(E, \sigma)$ is a distinguished $k$-torus and, by Lemma \ref{maxtorusina2}, ${\bf SU}(E, \tau) \hookrightarrow G \cong {\bf SU}(B, \sigma)$ over $k$.
\end{proof}
\noindent
For groups of type $F_4$ we have the following,
\begin{theorem}{\label{f5}}  Let $A$ be an Albert algebra over $k$ and $G= {\bf Aut}(A)$. Then $f_5(A) = 0$ if and only if $G$ contains a distinguished $k$-torus.
\end{theorem}
\begin{proof} Assume that $G$ contains a distinguished $k$-torus $T$. Then by Theorem \ref{isom}, $T$ is isotropic over an odd degree extension $M$ of $k$, hence $G$ becomes isotropic over $M$. Therefore $M$-rank of $G \otimes M \geq 1$ and $f_5(A \otimes M)= f_5(A) \otimes M= 0$ (Lemma \ref{isotropycond}). By Springer's theorem $f_5(A)=0$. Conversely, if $f_5(A)=0$, by (\cite{KMRT}, Prop. 40.7), $A \cong J(B,\sigma,u, \mu)$ for a central simple algebra $B$ over a quadratic \'etale extension $F$ of $k$, with a distinguished involution $\sigma$. Since $\sigma$ is distinguished, by Theorem {\ref{dist}} there exists a $k$-embedding of a distinguished $k$-torus $T$  in ${\bf SU}(B, \sigma)$. Now ${\bf SU}(B, \sigma) \hookrightarrow G$ over $k$ (see \S 2.4). Hence $T \hookrightarrow G$ over $k$ and $T$ is distinguished.
\end{proof}
\noindent
As a consequence of the above theorem, we have an alternative proof of ( \cite{NH}, Theorem $3.4$).
\begin{corollary}  Let $A$ be an Albert algebra over $k$ and $G= {\bf Aut}(A)$. Then $f_5(A) = 0$ if and only if there exists a $k$-embedding $ {\bf SU}( B, \sigma) \hookrightarrow G$ for some degree $3$ central simple algebra $B$ with center a quadratic \'etale $k$-algebra $F$ and with a distinguished involution $\sigma$.
\end{corollary}
\begin{proof} Suppose $ {\bf SU}( B, \sigma) \hookrightarrow G$ over $k$ for $(B, \sigma)$ as in the hypothesis. Since $\sigma$ is distinguished, by Theorem \ref{dist},  there exists a $k$-embedding $T \hookrightarrow {\bf SU}(B, \sigma)$ for a distinguished $k$-torus $T$. Hence $T \hookrightarrow {\bf SU}(B, \sigma)  \hookrightarrow G$ over $k$. Therefore, by Theorem \ref{f5}, $f_5(A)=0$. The proof of the converse follows exactly along the same lines as in the proof of Theorem \ref{f5}.
\end{proof}

\section {\bf Embeddings of rank-$2$ tori in $A_2$, $G_2$ and $F_4$}
It turns out that embeddings of unitary tori in groups of type $A_2$, $G_2$ and $F_4$ are intricately linked to the mod-$2$ invariants of these groups. We discuss this below. First we fix some terminology which will be used in the sequel.
\vskip1mm
\noindent
{\bf Groups arising from division algebras:} Let $G$ be a simple, simply connected $k$-group of type $A_2$. We will refer to $G$ as arising from a division algebra  if either $G \cong {\bf SU}(D, \sigma)$ for some degree $3$ central division algebra $D$ over a quadratic field extension $F$ of $k,$ with an involution $\sigma$ of the second kind or $G \cong {\bf SL_1}(D)$ for some degree $3$ central division algebra $D$ over $k$. Let $G$ be a $k$-group of type $F_4$. We will refer to $G$ as arising from a division algebra if $G \cong {\bf Aut}(A)$, where $A$ is an Albert division algebra over $k$. Let $G$ be a $k$-group of type $G_2$. We will refer to $G$ as arising from a division algebra if $G \cong {\bf Aut}(C)$, where $C$ is an octonion division algebra over $k$.

\begin{theorem}{\label{involution}} Let $G$ be a simple, simply connected group of type $A_2$ or $F_4$ defined over $k$, arising from a division algebra over $k$. Then,\\
\noindent
{\bf (1)} $G(k)$ contains no non-trivial involution over $k$.\\
{\bf (2)} There does not exists any rank-$1$ torus $T$ over $k$ such that $T \hookrightarrow G$ over $k$.\\
{\bf (3)} $G$ is $k$-anisotropic.\\
Moreover, these conditions hold over any field extension of $k$ of degree coprime to $3$.
\end{theorem}
\begin{proof} First we prove ${\bf (1)}$. Recall that an involution in a group is an element of order atmost $2$. Let $G$ be a simple, simply connected group of type $A_2$, arising from a division algebra $D$ over $k$. Let $Z(D)$ denote the center of $D.$ Then $[D: Z(D)]= 9.$ Let $\theta \in G(k)\subseteq D^*$ be an involution. Then $\theta^2=1$ and $N_D(\theta)=1$. Since $\theta^2=1$, $\theta$ generates the field extension $k(\theta)$ of $k$ of degree $\leq 2$ over $Z(D)$. Since the dimension $[D : Z(D)]= 9$, it follows that $\theta \in Z(D).$ Since $\theta^2=1$ and $\theta \in Z(D)$, $\theta= 1$ or $-1$ ($Z(D)$ is a field). Since $N_D(\theta)=1$ we have $\theta= 1$. Hence $G(k)$ does not contain any non-trivial involutions. When $G$ is a group of type $F_4$, the implication follows from a theorem of Jacobson ( \cite{J}, Chap. IX, Theorem 9). Moreover, let $M$ be any field extension of $k$ of degree coprime to $3.$ As seen above, if $G(M)$  contains a non-trivial involution, then $G \otimes M$ cannot arise from a division algebra. By Proposition \ref{coprime3}, $G$ cannot arise from a division algebra.
\vskip0.5mm
\noindent
We now prove ${\bf (2)}$. Suppose there exists a rank-$1$ torus $T$ over $k$ such that $T \hookrightarrow G$ over $k$. Necessarily, $T= K^{(1)}$, the norm torus of a quadratic extension $K/k$ (\cite{Vos}, Chap.II, \S IV, Example 6). But then $T$ splits over $K$, which in turn implies that $G$ becomes isotropic over $K$. Suppose $G$ is a group of type $F_4$ over $k$. Then $G= {\bf Aut}(A)$ for some Albert algebra $A$ over $k$. Since $G$ becomes isotropic over $K$, $A \otimes K$ is reduced (see Prop. \ref{red}). Hence $G$ does not arise from a division algebra over $k$, since no extension of degree coprime to $3$ can reduce a Albert division algebra (Proposition \ref{coprime3}). This is a contradiction. Now suppose $G$ is a group of type $A_2$ over $k$. Since $G$ becomes isotropic over $K$, by (\cite{T3},Table of Tits indices), $G \otimes K$ does not arise from a division algebra over $K$. By Proposition \ref{coprime3}, $G$ does not arise from a division algebra over $k$, a contradiction. Moreover, let $M$ be any field extension of $k$ of degree coprime to $3.$  Suppose there exists a rank-$1$ torus $T$ over $M$ such that $T \hookrightarrow G$ over $M.$ Then, as seen above, $G$ does not arise from a division algebra over $M$. Hence by Proposition \ref{coprime3}, $G$ does not arise from a division algebra over $k$. This is a contradiction. The proof of ${\bf (3)}$ follows from (\cite{T3}, Remark on Page 61, Table of Tits indices). 
\end{proof}

\begin{theorem}{\label{f3a}} Let $A$ be an Albert algebra over $k$ and $G= {\bf{Aut}}(A)$. Then the following are equivalent.
\vskip 1mm
\noindent
{\bf (a)} $f_3(A)=0$ (i.e, $Oct(G)$ is split).\\
\noindent
{\bf (b)} There exists a cubic \'etale $k$-algebra $L$ of trivial discriminant such that ${\bf L}^{(1)} \hookrightarrow G$ over $k$.\\
\noindent
{\bf (c)} ${\bf SL}_1(D) \hookrightarrow G$ over $k$, for a degree $3$ central simple algebra $D$ over $k$.
\end{theorem}
\begin{proof}
Let $f_3(A)=0$. Then, by (\cite{KMRT}, Prop. 40.5), $A$ is a first Tits construction and $A \cong J(D, \mu)$, where $ D$  is a degree $3$ central simple algebra over $k$. If $D$ is split, let $L= k \times k \times k$ and if $D$ is a division algebra, let $L$ a cubic cyclic extension of $k$ such that $L  \subseteq D_+$ (This is possible by Weddernburn's Theorem \cite{KMRT}, Pg. 303, 19.2). In either case, since ${\bf SL}_1(D) \hookrightarrow G$ (see \S 2.4), ${\bf L}^{(1)}  \hookrightarrow G$ over $k$. Hence  ${\bf (a)} \Rightarrow {\bf (b)}$ and  ${\bf (a)} \Rightarrow {\bf (c)}$ follows. \\
\noindent
For the proof of ${\bf (b)} \Rightarrow {\bf (a)}$, let ${\bf L}^{(1)} \hookrightarrow G$ over $k$, where $L$ is a cubic \'etale $k$-algebra of trivial discriminant. Clearly $L\cong k \times k \times k$ or $L$ is a cubic cyclic field extension of $k$. If  $L\cong k \times k \times k$ then ${\bf L}^{(1)} \cong \mathbb{G}_m \times \mathbb{G}_m$. Hence the $k$-rank of $G$ $\geq 2$ and, by Lemma \ref{isotropycond}, $f_3(A)= 0$. Let $L$ be a cubic cyclic field extension of $k$. Observe that ${\bf L}^{(1)} \otimes L \cong {\bf E}^{(1)}$, where $E= L \otimes L$. By Proposition \ref{disccubic}, $L \otimes L \cong L \times L \times L$ and hence ${\bf E}^{(1)}$ is an  $L$-split torus of rank-$2$, embedding in $G \otimes L$. Hence the $L$-rank of  $G \otimes L \geq 2$ and thus, by Lemma \ref{isotropycond}, $f_3(A \otimes L)=0$. By Springer's theorem, $f_3(A)=0$. We now prove ${\bf (c)} \Rightarrow {\bf (a)}$. Let ${\bf SL}_1(D) \hookrightarrow G$ over $k$, where $D$ is a degree $3$ central simple algebra over $k$. If $D$ is a division algebra, choose a cubic separable extension $L$ over $k$, $L \subseteq D$. Now,
$$ D \otimes_k L \cong M_3(L)\rm{~and~} {\bf SL}_1(D \otimes_k L) \cong {\bf SL}_3  \hookrightarrow G \otimes L.$$
Hence $G \otimes L$ has L-rank $\geqq$ $2$. By Lemma \ref{isotropycond}, $Oct(A)$ splits over $L.$ Since $[L:K]=3$, by Springer's theorem, $Oct (A)$ must split over $k$ and $f_3(A)= 0.$ In the case when $D$ is split, ${\bf SL}_1(D) \cong {\bf SL}_3 \hookrightarrow G$ over $k$. Hence $G$ is split over $k$ and $f_3(A)=0$. 
\end{proof}
\begin{lemma}{\label{useful}} Let $L= k \times K_0$ be a cubic \'etale algebra over $k$, where $K_0$ is a quadratic \'etale extension of $k.$ Let $K= k \times k$ and $T$ be the $K$-unitary torus associated with the pair $(L, K)$. Then $T \cong R_{K_0/k}(\mathbb{G}_m)$.
\end{lemma}
\begin{proof} By definition, $T= {\bf SU}(E, \tau)$, where $(E, \tau)= (L \otimes K, 1 \otimes ~\bar{}~)$. Note that
 $$(L \otimes K, 1 \otimes~\bar{}~) \cong (L \times L, \epsilon) \cong ((k \times K_0) \times (k \times K_0), \epsilon),$$ where $\epsilon: L \times L \mapsto L \times L$ is the switch involution. Hence 
$${\bf SU}(E,\tau)\cong{\bf SU}((k \times K_0) \times (k \times K_0), \epsilon).$$ 
For $((a,x),(b,y))\in (k \times K_0) \times (k \times K_0)$ we have, 
\vskip2mm
%\begin{eqnarray*}
$((a,x),(b,y))\epsilon((a,x),(b,y)) = (((a,x),(b,y))((b,y),(a,x)) = ((ab, xy),(ba,yx))$, and 
\vskip2mm 
$N_{E/K}((a,x),(b,y)) =  N_{(k \times K_0) \times (k \times K_0)/(k\times k)} = (a.N_{K_0/k}(x), b.N_{K_0/k}(y))$.
\vskip1mm
\noindent
%\end{eqnarray*}
Hence,
%\begin{eqnarray*}
 $$SU(E, \tau)\cong\{((a, x), (a^{-1}, x^{-1})) \in (k \times K_0) \times (k \times K_0)~| a.N_{K_0/k}(x)=1\}\cong K_0^*.$$
%\end{eqnarray*}
 From this it follows that ${\bf SU}(E, \tau) \cong R_{K_0/k}(\mathbb{G}_m)$.
\end{proof}
\begin{lemma}{\label{useful1}} Let $L= k \times k \times k$ and $K$ be a quadratic \'etale extension of $k$. Let $T$ be the $K$-unitary torus associated with the pair $(L, K).$ Then $T \cong  {\bf K}^{(1)} \times {\bf  K}^{(1)}$.
\end{lemma}
\begin{proof}  By definition, $T= {\bf SU}(E, \tau)$, where $(E, \tau)= (L \otimes K, 1 \otimes ~\bar{}~).$ It is immediate  that $(E, \tau) \cong (K \times K \times K, (~ \bar{}, ~ \bar{}, ~ \bar{}~))$. Hence,
$$SU(E, \tau) \cong \{ (x, y, z)\in K \times K \times K| x \overline{x}= y \overline{y}= z \overline{z}= 1, xyz=1\} \cong K^{(1)} \times  K^{(1)}.$$
It follows that ${\bf SU}(E, \tau) \cong  {\bf K}^{(1)} \times {\bf  K}^{(1)}$.
\end{proof}
\begin{theorem}{\label{KC}} 
{\bf (a)} Let $G$ be a $k$-group of type $G_2$ or a simply connected, simple group of type $A_2$. Let $L, K$ be \'etale algebras of dimension $3, 2$ resp. and $T$ be the $K$-unitary torus  associated with the pair $(L, K)$. Suppose there exists a $k$-embedding $T \hookrightarrow G$. Then $K \subseteq Oct(G)$.\\
{\bf (b)} If $G$ is a $k$-group of type $F_4$ or a simply connected, simple group of type $A_2$ arising from a division algebra and $T \hookrightarrow G$ over $k$, then $L$ must be a field extension.
\end{theorem}
\begin{proof} Let $(E, \tau)$ and $T$ be the $K$-unitary algebra and torus resp. associated with the pair $(L, K)$. By definition $T= {\bf SU}(E, \tau)$. For the assertion ${\bf (a)}$, we divide the proof into two cases.
\vskip0.5mm
\noindent
{\bf Case 1:}~ $L= k \times K_0$ for some quadratic \'etale extension $K_0$ of $k$.
 \vskip0.5mm
\noindent
Let $K= k \times k.$ By Lemma \ref{useful}, $T \cong R_{K_0/k}(\mathbb{G}_m) \hookrightarrow G.$ Therefore, the $k$-rank of $G \geq 1$. Thus by Lemma \ref{isotropycond}, $Oct(G)$ is split. When $K$ is a field extension, base changing to $K$ and applying the same argument, it follows that $Oct(G) \otimes K$ is split. Hence $K \subseteq Oct(G)$ (\cite{FG}, Lemma 5).
\vskip0.5mm
\noindent
{\bf Case 2:}~L is a field extension.
\vskip0.5mm
\noindent
Base changing to $L$, by Proposition \ref{disccubic}, we have, $L\otimes L \cong L \times K_0$ for $K_0= L \otimes \Delta$, where $\Delta$ is the discriminant algebra of $L$ over $k$. By case {\bf 1}, $K \otimes L \subseteq Oct(G) \otimes L$. Therefore if $K= k \times k$, $Oct(G) \otimes L$ is split and by Springer's theorem, $Oct(G)$ splits and  $K \subseteq Oct(G).$ Hence we may assume that $K$ is a field. Then $K \otimes L$ is a cubic field extension of $K$ and
$$(Oct(G) \otimes L) \otimes_L (L\otimes K) \cong Oct(G) \otimes L \otimes K \cong (Oct(G) \otimes K) \otimes_ K (K \otimes L)$$ is split, since $K \otimes L \subseteq Oct(G) \otimes L.$ Hence $(Oct(G) \otimes K)$ is split over the cubic extension $(K \otimes L)$ of $K$. Therefore by Springer's theorem, $Oct(G) \otimes K$ is split. Hence $K \subseteq Oct(G)$ (\cite{FG}, Lemma 5).
\vskip1mm
\noindent
Now we prove {\bf (b)}. Let $G$ be a $k$-group of type $F_4$ or $A_2$ as in the hypothesis and let $T \hookrightarrow G$ over $k$, where $T$ is the $K$-unitary torus associated to the pair $(L, K)$ as in the hypothesis. Assume that $L$ is not a field. Let $L= k \times K_0$ for some quadratic field extension $K_0$ of $k$. If $K= k \times k$ then, as in the proof of case ${\bf 1}$, $G$ is $k$-isotropic. Therefore, by Theorem \ref{involution}, $G$ cannot arise from a division algebra. Let $K$ be a field extension. By an easy calculation we see that $T \otimes K_0 = {\bf SU}(E \otimes K_0, \tau) \cong {\bf M}^{(1)} \times {\bf M}^{(1)}$, where $M= (K \otimes K_0)$. Note that $M^{(1)} \times M^{(1)}$ contains the involution $(-1, 1)$ defined over $K_0.$ Hence $G(K_0)$ contains a non-trivial involution. Therefore, by Theorem \ref{involution}, $G$ cannot arise from a division algebra. In the case when $L= k \times k \times k$, by Lemma \ref{useful1}, $T \cong {\bf K}^{(1)} \times  {\bf K}^{(1)}$. Again $K^{(1)} \times  K^{(1)} $ contains the involution $(-1, 1)$ defined over $k$. Hence $G(k)$ contains a non-trivial involution. Therefore, by Theorem \ref{involution}, $G$ cannot arise from a division algebra. Hence ${\bf (b)}$ follows.
\end{proof}
\begin{remark}{\label{countereg}} 
{\bf (1)} For $k$-groups of type $G_2$, {\bf (b)} fails to hold. To see this, let $L= k \times k \times k$ and $K$ be a quadratic field extension of $k.$ Let $T$ be the $K$-unitary torus associated with the pair $(L, K)$.
By Lemma \ref{useful1}, $T \cong  {\bf K}^{(1)} \times {\bf  K}^{(1)}$. Such a torus embeds in a $k$-group of type $G_2$ arising from a division algebra (see \cite{SV}, \S 2.1).
\vskip0.5mm
\noindent
{\bf (2)} For $k$-groups of type $F_4$, {\bf (a)} fails to hold. Let $C$ be an octonion division algebra over $k$. Let $\Gamma= diag(1,-1, -1) \in GL_3(k)$. Consider the reduced Albert algebra $A:= \mathcal{H}_3(C, \Gamma)$. Let $G= {\bf Aut}(A)$. Then $C= Oct(G)$ (see \S 2.3). Let $F \subseteq C$ be a quadratic subfield. By (\cite{PRT1}, \S 1, Thm. 1.1), there exists an isomorphism of Jordan algebras $  \mathcal{H}_3(C, \Gamma) \cong J(M_3(F), *_{\Gamma}, V, \mu)$, where $*_{\Gamma}(X)= \Gamma^{-1}\overline{X}^t\Gamma$, $V \in GL_3(F)$ with $*_{\Gamma}(V)= V$ and $det V= \mu \overline{\mu}$ for some $\mu \in F^*$. Let $L= k \times F$. Note that $L \subseteq M_3(F)$ as a $k$-subalgebra (via the embedding $(\gamma, x) \rightarrow diag(\gamma, x, x)$, $\gamma \in k$, $x \in F$). Since $*_{\Gamma}$ is a distinguished involution on $M_3(F)$ (\cite{HKRT}, Theorem 16), by (\cite{HKRT}, Cor. 18), it follows that $L \hookrightarrow (M_3(F), *_{\Gamma})_+$ over $k$. Let $T$ be the $F$-unitary torus associated with the pair $(L, F)$. Then $T \hookrightarrow {\bf SU}(M_3(F), *_{\Gamma}) \hookrightarrow G$ over $k$ (see \S 2.4). By case {\bf (ii)} of Theorem \ref{isom}, $T \cong R_{F/k}(\mathbb{G}_m)$. Hence $R_{F/k}(\mathbb{G}_m) \hookrightarrow G$ over $k$. Now consider $K= k \times k$. By Lemma \ref{useful}, ${\bf SU}(L \otimes K, \tau) \cong R_{F/k}(\mathbb{G}_m)$. Hence ${\bf SU}(L \otimes K, \tau) \hookrightarrow G$ over $k$ but $K$ does not embed in $C= Oct(G)$, since $C$ is a division algebra. 
\end{remark}
\noindent
However we have the following,
\noindent
\begin{theorem}{\label{KC1}} Let $G$ be a group of type $F_4$ defined over $k$. Let $K$ be a quadratic \'etale $k$-algebra and $L$ be a cubic \'etale $k$-algebra with trivial discriminant. Let $T$ be the $K$-unitary torus associated with the pair $(L, K)$. Suppose $T \hookrightarrow G$ over $k$. Then $K \subseteq Oct(G)$.
\end{theorem}
\begin{proof} Let $L$ be as in the hypothesis. When $K= k \times k$, we have $(L \otimes K, \tau) \cong (L \times L, \epsilon)$, where $\epsilon(x, y)= (y, x)$ for all $(x, y) \in L \times L$. Hence $T \cong {\bf L}^{(1)}$. By Theorem \ref{f3a}, $Oct(G)$ splits and hence $K \subseteq Oct(G)$. When $K$ is a field extension, base changing to $K$ we see that $Oct(G) \otimes K$ splits. Hence $K \subseteq Oct(G)$ (\cite{FG}, Lemma5). 
\end{proof}
\noindent
We now prove a factorization result for the mod-2 invariant $f_5(G)$ associated to an algebraic group $G$ of type $F_4$ defined over $k$, given an embedding of a rank-$2$ $K$-unitary torus in $G$. Let $L, K$ be \'etale algebras of dimension $3, 2$ resp. and let $T$ be the $K$-unitary torus associated with the pair $(L, K)$. Recall that with $T$, we associate the quadratic form $q_T:= <1, -\alpha \delta>$, where $Disc(L)= k(\sqrt{\delta})$ and $K= k(\sqrt {\alpha})$.\\
We need a classical factorization result for quadratic forms. Let $q$ be a quadratic form over a field $F$. Let $D_F(q)$ denote the set of values in $F^*$ represented by $q$ and $[D_F(q)]$ be the subgroups of $F^*$ generated by $D_F(q)$. When $q$ is a Pfister form $[D_F(q)]= D_F(q)$ (\cite{TYL}, Chap. X, Theorem 1.8).
\vskip1mm
\noindent
The following result on quadratic forms is essential for what follows.
\begin{theorem}{\label{subformfactor}} ( \cite{TYL}, Chap. IX, Pg. 305, Chap. X, Cor. 4.13)  For any quadratic form $\phi$ and any anisotropic quadratic form $\gamma$ over $k$, the following are equivalent,
\vskip0.5mm
\noindent
{\bf(i)} $ \phi \subseteq \gamma$ ( i.e, $\phi$ is isometric to a subform of the form $\gamma$ over $k$).\\
\noindent
{\bf(ii)} $ D_K(\phi) \subseteq D_K(\gamma)$ for any field $K \supseteq k$, where $D_K(\phi)$ denotes the set of values in $K^*$  represented by $\phi$. Moreover, if ${\phi}$ and ${\gamma}$ are both Pfister forms, then the above conditions are also equivalent to\\
{\bf(iii)} $  \gamma= \phi \otimes \tau$ for some Pfister form $\tau$ over $k$ (In this case we will call $\phi$ as a factor of $\gamma$).
\end{theorem}
\noindent
We can now prove:
\begin{theorem}{\label{albi}} Let $A$ be an Albert algebra over $k$ and $G= {\bf Aut}(A).$ Let  $K= k({\sqrt{\alpha}})$ be a quadratic \'etale $k$-algebra and $L$ be a cubic \'etale $k$-algebra with discriminant $\delta.$ Let $T$ be the  $K$-unitary torus associated with the pair $(L, K)$. Suppose $T \hookrightarrow G$ over $k$. Then $f_5(A)= q_{T} \otimes \gamma$ for some $4$-fold Pfister form $\gamma$ over $k$.
\end{theorem}
\begin{proof} Let $G= {\bf Aut}(A)$ be as in the hypothesis and let $T \hookrightarrow G$ over $k$.  
\vskip0.5mm
\noindent
{\bf{Claim}}: \hspace{0.1cm} $ D_M(q_T\otimes M) \subseteq  D_M( f_5( A) \otimes M)$ for all field extensions $M$ of $k$.
\vskip0.5mm
\noindent
 Let $F= k({\sqrt{{\alpha}{\delta}}})$. Then $ N_{F/k}= q_T$. Note that, $ N_{F/k} \otimes M=N_{F \otimes M/M }\rm{~and~} f_5(A) \otimes M=f_5(A \otimes M).$ If $N_{F \otimes M/M}$ is hyperbolic over $M$, then ${\alpha}= {\delta}~ M^{*^2}$ and hence $T \otimes M$ is a distinguished torus. Therefore, by Theorem \ref{f5}, $ f_5(A \otimes  M)= 0$ and the claim follows trivially. We may therefore assume both $N_{F \otimes M/M}$ and $f_5(A \otimes M)$ are anisotropic. Hence $K':= F \otimes M$ is a field extension of $M$. Now further base  changing to $K' \cong M {\otimes}_M K'$ we get, 
$$~~~~~~~~~~~~~~~~~~~~~~~~~~~~~~~~~~~(T \otimes M)  \otimes_{M} K' \cong T \otimes_{M} K' \hookrightarrow  {\bf Aut}  (  (A \otimes_{k}  M)  \otimes_{M}  K' ).~~~~~~~~~~~~~~~~~~~~~~~~~~~~~~~~~~(*)$$
Since ${\alpha}= {\delta}~ K'^{*^2}$, $T \otimes_{M} K'$ is a distinguished torus. Taking $K'$ as the base field and applying Theorem \ref{f5} to the embedding $(*)$ we get, $f_5((A \otimes_k M) \otimes_M K') = 0$.
\noindent
Now, since $K'$  over $M$ is a finite field extension and $f_5(A \otimes_k M) \otimes_M K'$ is split, we have, by Theorem (\cite{TYL}, Chap. VII, Cor. 4.4), $N_{{K'} / M}( {K'}^{*})\subseteq D_M(f_5(A \otimes_k M)).$ Since $N_{{K'} / M}( {K'}^{*})=  D_M(q_T \otimes M)$, we have $ D_M(q_T \otimes M ) \subseteq  D_M( f_5( A) \otimes M)$ for all extensions $M$ of $k$.
Hence by Theorem \ref{subformfactor}, $ N_{F/k}$ is isometric to a subform of $f_5(A)$ and we have, $f_5(A) = q_T \otimes \gamma$, for some 4-fold Pfister form $\gamma$ over $k$.
\end{proof}
\begin{remark} {\bf 1)} Note that the converse of the above theorem fails to hold.
\vskip0.5mm
\noindent
 Let $C$ denote the octonion division algebra represented by the $3$-fold (anisotropic) Pfister form  $<1,-x>\otimes<1,-y>\otimes<1,-z>$ over $k=\mathbb{C}(x,y,z,w)$. Let $F \subset C$ be a quadratic subfield and let $h=diag(h_1,h_2,h_3)$ denote the hermitian form on $F^\bot\subset C$ induced by the norm bilinear from (see \cite{NJ1}, \S 5, cf. Prop. \ref{hermitian}). Consider the Albert algebra $A:=J(M_3(F),*_h,1,\mu)$ where $*_h(X)=h^{-1}\overline{X}^th$ and $\mu\in F$ satisfies $\mu\overline{\mu}=1.$ Let $G= {\bf Aut}(A)$. Then $Oct(G)= C$ (\cite{PRT1}, \S 1, Theorem 1.1). By (\cite{NH}, Lemma 3.2), $ f_5(A)= n_C \otimes <<-1, -1>>.$ Since $-1$ is a square in $k$, we have $f_5(A)=0$. Let $K= k(\sqrt{w})$ 
and let $L$ be any cubic cyclic field extension of $k$. Let $T$ be the $K$-unitary torus associated with the pair $(L, K)$. Since $-w$ is not represented by $n_C$, $K \not \subset C$. Hence, by Theorem \ref{KC1}, $T$ cannot embed in $G$ over $k$, however $q_T$ divides $f_5(A)$.
\vskip0.5mm
\noindent
{\bf 2)}  Let $q_T$ be as in the hypothesis of Theorem \ref{albi}. Note that $q_T$ does not divide $f_3(G)$ in general. We use the construction as in the case {\bf (2)} of Remark \ref{countereg}. Let $C$ be an octonion division algebra. Let $\Gamma= diag(1,-1, -1) \in GL_3(k)$. Consider the reduced Albert algebra $A:= \mathcal{H}_3(C, \Gamma)$. Let $G= {\bf Aut}(A)$. Note that $Oct(G)= C$. Let $F \subseteq C$ be a quadratic subfield and $L= k \times F$. Let $T$ be the $F$-unitary algebra associated with the pair $(L, F)$. Then, as in the case {\bf (2)} of Remark \ref{countereg}, $T \hookrightarrow G$ over $k$. Since $Disc(L)= F$, we have $\alpha= \delta ~mod~{k^*}^2$. Hence the Pfister form $q_T= <1, -\alpha \delta> \cong <1, -1>$ and $q_T$ does not divide $f_3(G)$, since $C$ is a division algebra.
\end{remark}
\noindent
 On exactly similar lines we can derive a necessary condition for a rank-2 unitary torus to embed in a connected simple algebraic group of type $A_2$ or $G_2$:
\begin{theorem}{\label{fact1}}Let $G$ be a simple, simply connected $k$-group of type $A_2$ or $G_2$. Let $C:= Oct(G)$ and  $n_C$ denote the norm form of $C$. Let $K= k({\sqrt{\alpha}})$ be a quadratic \'etale $k$-algebra and $L$ be a cubic \'etale $k$-algebra with discriminant $\delta$. Let $T$ be the $K$-unitary torus associated with the pair $(L, K)$. Suppose there exists a $k$-embedding $T \hookrightarrow G$. Then $n_C= q_T \otimes \gamma$ for some two fold Pfister form $\gamma$ over $k$.
\end{theorem} 
\begin{proof} By Theorems \ref{g2to}, \ref{dist}, one sees that if $T$ is distinguished then $C$ splits. Now using same arguments as in the proof of Theorem \ref{albi}, we get the desired result.
\end{proof}
\begin{remark} Note that the converse of the above theorem fails to hold. 
\vskip0.5mm
\noindent
{\bf 1)} Let $*$  denote the unitary involution $*(X)= \overline{X}^t$ on $M_3(\mathbb{C})$ and let $G= {\bf SU}(M_3(\mathbb{C}), *)$. Let $C= Oct(G)$. Then $n_C= <1, 1> \otimes <1, 1> \otimes <1, 1>$ ( see \S 2.2). Hence $C$ is the unique octonion an division algebra over $\mathbb{R}$. Take $K= \mathbb{R} \times \mathbb{R}$ and $L= \mathbb{R} \times \mathbb{C}$. Let $T$ be the $K$-unitary torus associated with the pair $(L, K)$. Since $C$ is a division algebra, 
$K= \mathbb{R} \times \mathbb{R} \not \subset C$. Hence, by Theorem \ref{KC}, $T$ does not embed in $G$ over $\mathbb{R}$ but the quadratic form $q_T= <1, 1>$  associated  with $T$, is a factor of $n_C$.
\vskip0.5mm
\noindent
{\bf 2)} Let $G$ be a group of group of type $G_2$ over $k$ arising from an octonion division algebra $C$. Let $K_0= k(\sqrt{\delta}) \subset C$ be a fixed quadratic subfield. Note that $<1, -\delta>$ is a factor of $n_C$. Take $K= k \times k$ and $L= k \times K_0$. Let $T$ be the $K$-unitary torus  associated with the pair $(L, K)$. Since $C$ is a division algebra, we have $K \not \subset C$. Hence, by Theorem \ref{KC}, $T$ does not embed in $G$ over $k$, but the quadratic form $q_T= <1, -\delta>$  associated  with $T$, is a factor of $n_C$
\end{remark}
\noindent
Let $A$ be an Albert algebra over $k$ and $G= {\bf Aut}(A)$. Let $L, K$ be \'etale algebras of dimension $3, 2$ resp. and $T$ be the $K$-unitary torus associated with the pair $(L, K)$. By case {\bf(2)} of Remark \ref{countereg}, if there is a $k$-embedding $T \hookrightarrow G$, then $K$ need not embed in $Oct(G)$, i.e. if  $K= k({\sqrt{\alpha}})$ then $<1, -\alpha>$ is not a factor of $f_3(G)$ in general. However,
\begin{theorem}{\label{albi1}} Let $A$ be an Albert algebra over $k$ and $G= {\bf Aut}(A)$. Let  $K= k({\sqrt{\alpha}})$ be a quadratic \'etale $k$-algebra and $L$ be a cubic \'etale $k$-algebra. Let $T$ be the $K$-unitary torus  associated with the pair $(L, K)$. Suppose there exists a $k$-embedding $T \hookrightarrow G$. Then $f_5(A)= <1, -{\alpha}>\otimes \gamma$ for some $4$-fold Pfister form $\gamma$ over $k$.
\end{theorem}
\begin{proof} We first assume that $K\cong k \times k.$ If $L$ is not a field, then $L= k \times K_0$, where $K_0$ is a quadratic \'etale $k$-algebra. By Lemma \ref{useful}, $T \cong R_{K_0/k}(\mathbb{G}_m) \hookrightarrow G$. Therefore $T$ is $k$-isotropic and $k$-rank of $G \geq 1$. Hence by Lemma \ref{isotropycond}, $f_5(A)= 0$. Let $L$ be a field extension. Base changing to $L$ we have, $L\otimes L \cong L \times K_0$, where $K_0= L \otimes \Delta$ and $\Delta$ is the discriminant algebra of $L$ over $k$. Hence by the above argument, $T \otimes L$ is $L$-isotropic and $f_5(A \otimes L)= 0$. By Springer's theorem $f_5(A)= 0$. Therefore, if $\alpha \in {k^*}^2$ (i.e, $K= k \times k$) and $T \hookrightarrow G$ over $k$, then $f_5(A)=0$. Using the same arguments as in Theorem \ref{albi}, with Pfister form $<1, -\alpha>$ instead of $q_T$, we get the desired result.
\end{proof}
\begin{theorem} Let $G$ be a simple, simply connected algebraic group defined over $k$. Let $L$ be a cubic \'etale $k$-algebra with discriminant $K_0$. Suppose there exists an $k$-embedding ${\bf L}^{(1)} \hookrightarrow G$. We then have:\vskip0.5mm
\noindent
{\bf (a)} if $G$ is of type $G_2$ or $A_2$ then $Oct(G)$ splits. \\
{\bf (b)} if $G$ is of type $F_4$ then $f_5(G)=0$ and $K_0 \subset Oct(G)$.
\end{theorem} 
\begin{proof} Let $L$ be as in the hypothesis and $K= k \times k$. Let $(E, \tau)$ and $T$ be the $K$-unitary algebra and torus resp. associated with the pair $(L, K)$.  We have $(E, \tau) \cong (L \times L, \epsilon)$, where $\epsilon(x, y)= (y, x)$ for all $(x, y) \in L \times L$. Hence $T \cong {\bf L}^{(1)}$. Therefore $T \hookrightarrow G$ over $k$. 
\vskip0.5mm
\noindent
{\bf (a)}~Let $G$ be a simply connected, simple group of type $G_2$ or $A_2$. Then, by Theorem \ref{KC}, $K \subseteq Oct(G)$. Since $K$ is split, $Oct(G)$ splits.
\vskip0.5mm
\noindent
{\bf (b)}~Let $G$ be a $k$-group of type $F_4$. Then by Theorem \ref{albi1}, $f_5(A)= <1, -{\alpha}>\otimes \gamma$ for some $4$-fold Pfister form $\gamma$ over $k$ where $K= k(\sqrt{\alpha})$. Since $K= k \times k$, we have $\alpha \in {k^*}^2$ and hence $f_5(A)= 0$. If $K_0$ is split then $L$ has trivial discriminant. Hence, by Theorem \ref{f3a}, $f_3(G)=0$ and $Oct(G)$ splits. Therefore $K_0 \subset Oct(G)$. If $K_0$ is a field extension, base changing to $K_0$ we see that $L \otimes K_0$ is a cubic \'etale algebra over $K_0$ of trivial discriminant. Applying Theorem \ref{f3a} to the $K_0$-embedding ${\bf L}^{(1)}\otimes K_0 \hookrightarrow G \otimes K_0$, we get $f_3(G \otimes K_0)= f_3(G) \otimes K_0= 0$. Hence $Oct(G)$ splits over $K_0$ and thus $K_0 \subseteq Oct(G)$.
\end{proof}
\begin{remark} Let $G$ be a group of type $F_4$ over $k$. Let $L$ be a cubic \'etale $k$-algebra. Suppose there exists an $k$-embedding ${\bf L}^{(1)} \hookrightarrow G$. Then $f_3(G)$ may not be zero. We use the construction as in the case {\bf (2)} of Remark \ref{countereg}. Let $C$ be an octonion division algebra. Let $\Gamma= diag(1,-1, -1) \in GL_3(k)$. Consider the reduced Albert algebra $A:= \mathcal{H}_3(C, \Gamma)$. Let $G= {\bf Aut}(A)$. Note that $Oct(G)= C$. Let $F \subseteq C$ be a quadratic subfield and $L= k \times F$. Let $T$ be the $F$-unitary torus associated with the pair $(L, F)$. As in the case {\bf (2)} of Remark \ref{countereg}, $T \hookrightarrow G$. Note that $T \cong {\bf SU}(( k \times F) \otimes F,~ 1 \otimes ~\bar{}~) \cong R_{K/k}(\mathbb{G}_m) \cong {\bf L}^{(1)}$ (Theorem \ref{isom}, case {\bf (ii)}). Hence ${\bf L}^{(1)} \hookrightarrow G$ but $f_3(G) \neq 0$.

\end{remark}
\subsection{Cohomology of unitary tori}
In this section we will use some definitions and results from (\cite{KMRT}, \S 29). Fix a separable closure $k_{sep}$ of $k$ and let $\Gamma= Gal(k_{sep}/k)$. Let $G$ be an algebraic group defined over $k$ and let $\rho: G \longrightarrow {\bf GL}(W)$ be a representation for a finite dimensional $k$-vector space W. Fix an element $w \in W$ and identify $W$ with a $k$-subspace of $W_{sep}= W \otimes k_{sep}$. An element $w' \in W_{sep}$ is called a twisted $\rho$-forms of $w$ if $w'= \rho_{sep}(g)(w)$ for some $g \in G(k_{sep})$, where $\rho_{sep}= \rho \otimes k_{sep}$.
Let ${\overline{A}(\rho, w)}$ denote the groupoid whose objects are the twisted $\rho$-form of $w$ and whose morphisms $w' \rightarrow w''$ are the elements $g \in G(k_{sep})$ such that $\rho_{sep}(g)(w')= w''$. Let $ {A(\rho, w)}$ denote the groupoid whose objects are the twisted $\rho$-forms of $w$ which lie in $W$ and morphisms $w' \rightarrow w''$ are the elements $g \in G(k)$ such that $\rho(g)(w')= w''$. Let $X$ denote the $\Gamma$-set of objects of $\overline{A}(\rho, w)$. Then $X^{\Gamma}$ is the set of objects of $A(\rho, w)$. Also, the set of orbits of $G(k)$ in  $X^{\Gamma}$ is the set of isomorphism classes $Isom(A(\rho, w))$ of objects of $A(\rho, w)$. Let $w' \in A(\rho, w)$. By $[w'] \in Isom(A(\rho, w))$ we will denote the  isomorphism class of $w'$. Let ${\bf Aut}_G(w)$ denote the stabilizer of $w$ in $G$.
\begin{proposition}(\cite{KMRT}, Prop. 29.1){\label{cohomo}} If $H^1(k, G)= 0$, there is a natural bijection of pointed sets
$$Isom(A(\rho, w)) \leftrightarrow H^1(k, {\bf Aut}_G(w))$$
which maps the isomorphism class of $w$ to the base point of $H^1(k, {\bf Aut}_G(w))$. 
\end{proposition}
\noindent
The bijection between these sets is as follows: for $w' \in A(\rho, w),$ choose $g \in G(k_{sep})$ such that $ \rho_{sep}(g)(w)= w'$. Define a 1-cocycle class $[\alpha_{\sigma}] \in H^1(k, {\bf Aut}_G(w))$ by $\alpha_{\sigma}= g^{-1}\sigma(g)$. Conversely let $[\alpha] \in H^1(k, {\bf Aut}_G(w))$. Since $H^1(k, G)= 0$, $\alpha= g^{-1}\sigma(g)$ for some $g \in G(k_{sep})$. The corresponding object in $A(\rho, w)$ is $ \rho_{sep}(g)(w)$.\\
\noindent
We now specialize and adapt some of the computations done in (\cite{KMRT}, \S 29.17), to the case of unitary algebras.
Let $k$ be a field and $K$ be a quadratic \'etale  $k$- algebra. Let ~$\bar{}$~ denote the non-trivial $k$-automorphism of $K.$ Let $L$ be an \'etale algebra of dimension $n$ over $k$ and $E=  L \otimes K$ be the associated $K$-unitary algebra with the involution $\tau= 1\otimes \bar{}$.\\
\noindent
We first calculate  $H^1(k, {\bf SU}(E, \tau))$. Let $W= E \oplus K$. Define a representation $\rho: {\bf GL}_1(E) \longrightarrow {\bf GL}(W)$ by, 
$$ \rho(b)(x,y)= (bx \tau(b), N_{E \otimes k_s/K \otimes k_s }(b)y) $$ for all $b \in {\bf GL}_1(E)$, $x \in E$, $y \in K$.
Let $w_0= (1,1) \in W$. Note that ${\bf GL}_1(E)(k_s)= (E\otimes k_s)^*$.
\vskip0.5mm
\noindent
{\bf{Claim}}: \hspace{0.2cm} ${\bf Aut}_{{\bf GL}_1(E)}(w_0)= {\bf SU}(E, \tau)$.\\
We have,
\begin{eqnarray*}
 {\bf Aut}_{{\bf GL}_1(E)}(w_0) & = & \{g \in  (E\otimes k_s)^*| ~\rho_{sep}(g)(1,1)= (1,1)\}\\
& = & \{g \in  (E\otimes k_s)^*|~ g\tau(g)=1,~ N_{E \otimes k_s/K \otimes k_s}(g)=1\}\\
& = & {\bf SU}(E, \tau).
\end{eqnarray*}
Hence, in view of Proposition \ref{cohomo}, we have a bijection $$ \eta: Isom(A(\rho, w_0)) \leftrightarrow H^1(k, {\bf SU}(E, \tau)).$$ 
We define a  product on $Isom(A(\rho, w_0))$ as follows: $$\rho_{sep}(g)(w_0) \rho_{sep}(g')(w_0):= \rho_{sep}(gg')(w_0) ~\text{for}~ \text{all}~ g, g' \in (E\otimes k_{sep})^*.$$ A routine calculation shows that this product is well defined. Since ${\bf SU}(E, \tau)$ is a torus, $H^1(k, {\bf SU}(E, \tau))$ is an abelian group. It is immediate that $\eta$ is a homomorphism of groups. Define $$V:= \{ (s,z) \in L^* \times K^*| ~N_{L/k}(s)= z\overline{z}\}.$$
Given a twisted $\rho$-form $w'$ of $w_0$ which lies in $W$, there exists  $b \in (E\otimes k_{sep})^*$ such that $w'=\rho_{sep}(b)(w_0)$. Now $\rho_{sep}(b)(w_0)= \rho_{sep}(b)(1, 1)= (b\tau(b), N_{E\otimes k_{sep}/(K\otimes k_{sep})}(b))$. Along similar lines as in (\cite{KMRT}, \S 29.17), we can show that $\rho_{sep}(b)(w_0) \in V$ and $V$ is precisely the set of twisted $\rho$-forms of $w_0$ which lie in $W$. Define an equivalence ${\sim}$ on $V$ as follows:
$$(s, z)\sim (s', z')~ \text{if}~ \text {and}~ \text {only}~\text{if}~ s'= bs\tau(b) ~\text {and}~ z'= N_{E/K}(b)z ~\text {for}~ \text {some}~ b \in E^*.$$ We will denote equivalence class of $(s, z) \in V$ by $[(s, z)]$. Note that $V$ is a subgroup of $L^* \times K^*$. It is easy to see that the product on $V$ induces a well defined product on $V/ \sim$ as follows: $$[(s, z)][(s', z')]= [(ss', zz')]~ \text{for}~\text{all} ~(s, z), (s', z') \in V.$$
\noindent
 Define $\xi: Isom(A(\rho, w_0)) \rightarrow V/\sim $ by $\xi([w'])= [(b\tau(b), N_{E\otimes k_{sep}/(K\otimes k_{sep})}(b))]$, where $w'= \rho_{sep}(b)(w_0)$, for some $b \in (E \otimes k_{sep})^*$. It follows that $\xi$ is a homomorphism of groups. \\
\noindent
We have proved the following,
\begin{theorem} Let $L, K$ be \'etale $k$-algebras of dimension $n, 2$ resp. and $T$ be the $K$-unitary torus associated with the pair $(L, K)$. Then there exists a natural isomorphism: $H^1(k, T) \mapsto V/ \sim$ of groups.
\end{theorem}
\noindent
Henceforth we will identity $H^1(k, T)$ with $V/ \sim$ and write elements in $H^1(k, T)$ as equivalence classes $[(s, z)] \in V/ \sim$.
\begin{theorem}{\label{cohomo1}}  Let $L, K$ be \'etale $k$-algebras of dimension $n, 2$ resp. Let $E$ be the $K$-unitary algebra and $T$ be the $K$-unitary torus associated with the pair $(L, K)$. Then,  $$\frac{H^1(k, ~T)}{K_0^{(1)}}\cong \frac{S}{ N_{E/L}(E^*)},$$ where
 $$S:= \{u \in L^*| ~N_{L/k}(u) \in N_{K/k}(K^*)\}~\rm{and} ~K_0^{(1)}:= \{[(1, \mu)] \in H^1(k, ~T) |~ \mu\overline{\mu}=1\}.$$
\end{theorem}
\begin{proof} By definition, $T= {\bf SU}(E, \tau)$, where $(E, \tau)= (L \otimes K, 1 \otimes ~\bar{}~)$ is the  $K$-unitary algebra associated with the pair $(L, K)$. Define $\phi: H^1(k, {\bf SU}(E, \tau)) \longrightarrow \frac{S}{N_{E/L}(E^*)}$ by $\phi([(s, z)])= sN_{E/L}(E^*).$ If $(s, z)\sim (s', z')$ then $s = s' b\tau(b)$ for some $b \in E$. Hence $s= s' N_{E/L}(E^*)$ and $\phi$ is well defined. We now check that $\phi$ is surjective. Let $s \in S$. By definition, there exists $z \in K^*$ such that $N_{L/k}(s)= z\overline{z}$, for some $z \in K$. Hence $\phi([(s, z)])= s$, showing that $\phi$ is onto. Clearly $\phi$ is a homomorphism. Now, $Ker~{\phi}= \{ [(s, z)]|~ s\in N_{E/L}(E^*)\}.$ Clearly, $K_0^{(1)} \subseteq Ker~{\phi}$. Let $[(s, z)]\in Ker~{\phi}.$ Then $s \in N_{E/L}(E^*)= e\tau(e)$ for some $e\in E.$ Let $\mu= zN_{E/K}(e^{-1})$. Then $$N_{L/k}(s)= z\overline{z}= N_{E/K}(s)= N_{E/K}(e\tau(e))= N_{E/K}(e)\overline{N_{E/K}(e)}.$$ Hence $z\overline{z}= N_{E/K}(e)\overline{N_{E/K}(e)}.$ Therefore $\mu\overline{\mu}=1$. It follows that $(s, z)= (e\tau(e), z) \sim (1, \mu)$. Hence, $Ker{\phi}= K_0^{(1)}.$
\end{proof}
\noindent
We obtain below an explicit expression for $H^1(k, {\bf SU}(E, \tau)).$ Consider the the exact sequence,
$$1\longrightarrow K_0^{(1)}\xrightarrow{q}  H^1(k, {\bf SU}(E, \tau))\xrightarrow{\phi} \frac{S}{N_{E/L}(E^*)}\longrightarrow 1,$$
where $q$ denotes the inclusion map and $\phi$ is as above.
We provide a splitting of this sequence when dimension of $L$ is odd. We will from here on assume that the $k$-dimension $n$ of $L$ is odd. Let $n= 2r+1$.\\
\noindent
Define $t:  H^1(k, {\bf SU}(E, \tau))\longrightarrow K_0^{(1)}$ by,
$$t([(u, \mu)])= [(1, \mu^{-r}{\overline{\mu}}^{r})].$$
We first check that this map is well defined. Let $w \in E^*$. Then $(u, \mu) \sim (wu\tau(w), N_{E/K}(w)\mu)$. Now $w^{-r}{\tau(w)}^{r} \in U(E, \tau)$. Hence
 $$(1, \mu^{-r}{\overline{\mu}}^{r}) \sim (1, N_{E/K}(w^{-r}{\tau(w)}^r)\mu^{-r}{\overline{\mu}}^{r})= (1, (N_{E/K}(w)\mu)^{-r}{\overline{(N_{E/K}(w)\mu)^r}}).$$ Therefore, $t$ is well defined. It is immediate that $t$ is a homomorphism. We have $$t\circ~ q[(1,~\mu)]= [(1, \mu^{-r}{\overline{\mu}}^r)]= [(1, \mu)]$$ (Since ${\mu}^{2r+1}= N_{E/K}(\mu)$ and $\mu\overline{\mu}= 1$). Hence, $t\circ q= Id_{K_0^{(1)}}.$
Therefore there exists a homomorphism $\psi: \frac{S}{N_{E/L}(E^*)} \longrightarrow H^1(k, {\bf SU}(E, \tau))$ such that $\phi \circ \psi= Id$. In fact $\psi$ is given by, 
$$\psi([u]):= [(u, ~\mu)]q({t([(u,~ \mu)])}^{-1})= [(u,~ \mu^{r+1}{\overline{\mu}^{-r}})],$$ where $N_{L/k}(u)= \mu\overline{\mu}$, $\mu \in K.$ We now make some observations based on the above exact sequence. We have,
\begin{eqnarray*}
Ker~ t &= &\{[(u,~ \mu)]|~(1,~\mu^{-r}{\overline{\mu}}^r)\sim (1,1)\}=  \{[(u,~ \mu)]|~ \mu^{-r}{\overline{\mu}}^r\in N_{E/K}(U(E, \tau)) \}.\\ 
& = & Image~ \psi= \{[(u, ~{\mu}^{r+1}{\overline{\mu}}^{-r})]|~ N_{L/k}(u)= \mu\overline{\mu}\}.
\end{eqnarray*}
Since $\psi$ is an injective homomorphism, we have $Image~\psi \cong \frac{S}{N_{E/L}(E^*)}$. Hence 
$$\frac{S}{N_{E/L}(E^*)} \cong   \{[(u,~ \mu)]|~ \mu^{-r}{\overline{\mu}}^r\in N_{E/K}(U(E, \tau)) \}\cong \{[(u, ~\mu^{r+1}{\overline{\mu}}^{-r})]|~ N_{L/k}(u)= \mu\overline{\mu}\}.$$
Owing to the splitting of the exact sequence above we have, $H^1(k, {\bf SU}(E, \tau))= Image~q \times Ker~t$. We have already seen that $Ker~t \cong \frac{S}{N_{E/L}(E^*)}$. Now, $Image~q= K_0^{(1)}$. Let $K^{(1)}$ denote the norm $1$ elements of $K$. Define a map $$\chi: K^{(1)} \longrightarrow K_0^{(1)}$$ by
$\chi(\mu):= [(1, \mu)]$ for all $\mu \in K^{(1)}$. This map is clearly a surjective homomorphism.
 Now,
\begin{eqnarray*}
Ker~\chi &= &\{\mu \in K^{(1)}|~[(1,~\mu)]= [(1,~1)]\}.\\
& = &\{\mu \in K^{(1)}|~\mu= N_{E/K}(w),~ w\tau(w)=1,~w \in E\}.\\
& = &N_{E/K}(U(E, \tau)). 
\end{eqnarray*}
Hence, $\frac{K^{(1)}}{N_{E/K}(U(E, \tau))}\cong K_0^{(1)}.$ We summarize this as:
\begin{theorem}{\label{dcps}} Let $K$ be a quadratic \'etale $k$-algebra and $L$ be an \'etale $k$-algebra of dimension  $n=2r+1$. Let $E$ be the $K$-unitary algebra and $T$ the $K$-unitary torus associated with the pair $(L, K)$. Then,
$$H^1(k, T) \cong \frac{K^{(1)}}{N_{E/K}(U(E, \tau))} \times  \frac{S}{ N_{E/L}(E^*)}.$$
\end{theorem}
\noindent
In fact, an explicit isomorphism is as follows:
$$\phi: H^1(k, T) \longrightarrow K_0^{(1)} \times  \frac{S}{ N_{E/L}(E^*)}$$
$$\phi([(u,~ \mu)])= ([(1,~{\mu}^{-r}{\overline{\mu}}^r)],~ [u]).$$ 
\noindent
We now prove a somewhat analogous result to Theorem \ref{cohomo1}, for the cohomology of a unitary torus.
\begin{theorem}{\label{definem}}  Let $L, K$ be \'etale $k$-algebras of dimension $n, 2$ resp. Let $E$ be the associated $K$-unitary algebra and $T$ the $K$-unitary torus associated with the pair $(L, K)$. Then,
 $$\frac{H^1(k, T)}{L_0^{(1)}} \cong \frac{M}{N_{E/K}(E^*)},$$ where
 $$M= \{\mu \in K^*|~\mu\overline{\mu}\in N_{L/k}(L^*)\}~ and ~L_0^{(1)}= \{[(u,~1)]|~N_{L/k}(u)=1\}.$$
\end{theorem}
\begin{proof} By definition, $T= {\bf SU}(E, \tau)$, where $(E, \tau)= (L \otimes K, 1 \otimes ~\bar{}~)$ is the $K$-unitary algebra associated  with the pair $(L, K)$.
We define a map
$$\phi: H^1(k, {\bf SU}(E, \tau)) \longrightarrow \frac{M}{N_{E/K}(E^*)}$$ by
$[(s, z)]  \mapsto zN_{E/K}(E^*)$. It is easy to see that $\phi$ is a well defined surjective homomorphism and $Ker~ {\phi}= \{ [(s, z)]|~ z\in N_{E/K}(E^*)\}$. Clearly, $L_0^{(1)} \subseteq Ker~{\phi}$. Let $[(s, z)]\in Ker~{\phi}$. Then $z= N_{E/K}(w)$, for some $w\in E^*$. Let $u= w^{-1}s\tau(w^{-1})$. Now, $N_{E/K}(u)= N_{E/K}(w^{-1}\tau(w^{-1})s)$ and
$$N_{E/K}(s)= N_{L/k}(s)= z\overline{z}= N_{E/K}(w\tau(w)).$$ Hence $N_{E/K}(u)=1$. Also $(u, 1)= (w^{-1}s\tau(w^{-1}), 1) \sim (s,~z)$. Therefore we have $Ker~{\phi}= L_0^{(1)}$
\end{proof}
\noindent
We now provide a decomposition of $H^1(k, {\bf SU}(E, \tau))$ analogous to that in Theorem \ref{dcps}. Consider the exact sequence,
$$1\longrightarrow L_0^{(1)}\xrightarrow{q}  H^1(k, {\bf SU}(E, \tau)) \xrightarrow{\phi} \frac{M}{N_{E/K}(E^*)}\longrightarrow 1,$$
where the maps $q$ and $\phi$ are as above.
We provide a splitting of this sequence when dimension of $L$ is odd. We will from here on assume that the $k$-dimension $n$ of $L$ is odd. Let $n= 2r+1$.\\
We define a map $$t:  H^1(k, {\bf SU}(E, \tau))\longrightarrow L_0^{(1)}$$ by
$t([(u, \mu)]):= [(u^nN_{L/k}(u^{-1}),~1)]$.
We first check that this map is well defined. Let $w\in E^*$. Then $(u, \mu) \sim (wu\tau(w),~ N_{E/K}(w)\mu)$. Since $N_{E/K}(w^nN_{E/K}(w^{-1}))= 1$ and $N_{E/K}(w^{-1}\tau(w^{-1}))= N_{L/k}(w^{-1}\tau(w^{-1}))$ we have,
\begin{eqnarray*}
(u^nN_{L/k}(u^{-1}), 1) & \sim & (w^nN_{E/K}(w^{-1})u^nN_{L/k}(u^{-1})\tau(w^n)N_{E/K}(\tau(w^{-1})),1)\\
& = & ({(wu\tau(w))}^nN_{L/k}({{(wu\tau(w))}^{-1}}),~1)).
\end{eqnarray*}
Therefore, $t$ is well defined. It is easily checked that $t$ is a homomorphism.
Since $u\in L$ and $N_{E/K}(u)= N_{L/k}(u)= u\tau(u)= 1$, we have, $$t\circ~ q[(u,~1)]= [(u^{2r+1},~1)]= [(u^{2r-1}(u\tau(u)), 1)] = [(u^{2r-1}, 1)].$$ By a similar calculation,  $[(u^{2r-1}, 1)]=  [(u^{2r-3}, 1)]=...= [(u, 1)]$. Hence $[(u^n, 1)]= [(u, 1)]$ and therefore $t\circ q= Id_{L_0^{(1)}}$.\\
Hence there exists a homomorphism $\psi: \frac{M}{N_{E/K}(E^*)} \longrightarrow H^1(k, {\bf SU}(E, \tau))$ such that $\phi \circ \psi= Id$. Explicitly, this map is given by
$$\psi([\mu]):= [(u, ~\mu)]q(t([(u^{-1},~ \mu^{-1})]))= [(u^{-n+1}N_{L/k}(u),~ \mu)]$$ where $N_{L/k}(u)= \mu\overline{\mu}$, $\mu \in K^*$.
We now make some observations based on this exact sequence. We have,
\begin{eqnarray*}
Ker~ t & = &\{[(u,~ \mu)]|~[(u^nN_{L/k}(u^{-1}),~1)]\sim (1,1)\}\\
& = & Image~ \psi\\
& = & \{[(u^{-n+1}N_{L/k}(u), ~\mu)]|~ N_{L/k}(u)= \mu\overline{\mu}\}.
\end{eqnarray*}
Since $\psi$ is an injective group homomorphism, $Image~\psi \cong \frac{M}{N_{E/K}(E^*)}$. Hence 
$$\frac{M}{N_{E/K}(E^*)} \cong  \{[(u,~ \mu)]|~[(u^nN_{L/k}(u^{-1}),~1)]\sim (1,1)\}  \cong \{[(u^{-n+1}N_{L/k}(u), ~\mu)]|~ N_{L/k}(u)= \mu\overline{\mu}\}.$$
Since the above sequence is split exact, we have $H^1(k, {\bf SU}(E, \tau))= Image~q \times Ker~t$. We have already seen that $Ker~t \cong \frac{M}{N_{E/K}(E^*)}$ and $Image~q= L_0^{(1)}$. Let $L^{(1)}$ denote norm $1$ elements of $L$ and $E^{(1)}= \{x\in E|~N_{E/K}(x)=1\}$. Now define
$\phi: L^{(1)} \longrightarrow L_0^{(1)}$ by $u \mapsto [(u,1)].$ It is easily checked that $Ker~ \phi= N_{E/L}{(E^{(1)})}$. Hence $ \frac{L^{(1)}}{N_{E/L}{(E^{(1)})}} \cong L_0^{(1)}.$ We  summarize this as, 
\begin{theorem}{\label{h1}} 
Let $K$ be a quadratic \'etale $k$-algebra and $L$ be an \'etale $k$-algebra of dimension $n=2r+1$. Let $E$ be the  $K$-unitary algebra and $T$ the $K$-unitary torus associated to the pair $(L, K)$. Then,
$$H^1(k, T) \cong \frac{L^{(1)}}{N_{E/L}{(E^{(1)})}} \times  \frac{M}{ N_{E/K}(E^*)}.$$
\end{theorem}
\noindent
We now discuss the special case, when $L= k \times K_0$, where $ K_0$ is a quadratic \'etale extension of $k$. 
\begin{theorem} {\label{split}} Let $K$ and $K_0$ be a quadratic \'etale extensions of $k$. Let $L= k\times K_0$  and $T$ be the $K$- unitary torus associated with the pair $(L, K)$. Then $H^1(k, T)= K_0^*/N_{K \otimes K_0/K_0}(K_0^*).$
\end{theorem}
\begin{proof} By definition, $T= {\bf SU}(E, \tau)$, where $(E, \tau)= (L \otimes K, 1 \otimes ~\bar{}~)$ is the $K$-unitary algebra associated with the pair $(L, K)$. Let $M:= K \otimes K_0$. Then $L \otimes K \cong K \times (K \otimes K_0)= K \times M$. Now,
\begin{eqnarray*}
SU(E, \tau) & \cong & \{(a, x) \in K \times M|~a\overline{a}=1, x \tau(x)=1, aN_{M/K}(x)=1\}\\
& = & \{(N_{M/K}(x^{-1}), x)|~x \tau(x)=1, x \in M^*\}.
\end{eqnarray*}
It follows that ${\bf SU}(E, \tau) \cong R_{K_0/k}^{(1)}(M)$. By Shapiro's Lemma (\cite{KMRT}, Lemma 29.6), 
\vskip1mm
\noindent
$ H^1(k, R_{K_0/k}^{(1)}(M)) = H^1(K_0, {\bf M}^{(1)}).~\text{Hence},~H^1(k,T)= K_0^*/N_{K \otimes K_0/K_0}(K_0^*).$
\end{proof}
\begin{corollary}{\label{uni}} Let $K= k(\sqrt{\alpha})$ and $K_0$ be quadratic field extensions of $k$. Let $L= k\times K_0$  and $T$ be the $K$- unitary torus associated with the pair $(L, K)$. Then $H^1(k, T)= 0$ if and only if the quadratic form $<1, -\alpha>$ becomes universal over $K_0$.
\end{corollary}
\begin{corollary} Let $K$ be a quadratic \'etale extension of $k$ and $L= k \times K_0$, where $K_0$ is a quadratic \'etale extension of $k$. Let $T$ be the $K$-unitary torus associated with the pair $(L, K)$. Let $H^1(k, T)=0$. Then any composition algebra of dimension $\geq 4$ which contains $K$ contains $K_0$. 
\end{corollary}
\begin{proof} By Theorem \ref{split}, $H^1(k, T)=0$ if and only if $N_{K \otimes K_0/ K_0}(K \otimes K_0)^*= K_0^*$.
 Let $C$ be a composition algebra properly containing $K$. Then $K \otimes K_0 \subseteq C \otimes K_0$. By doubling, $C \otimes K_0= (K \otimes K_0) \oplus (K \otimes K_0).x$, for some $x \in (K \otimes K_0)^{\perp}$, $N_{C \otimes K_0}(x) \neq 0$. But since $K_0^*= N_{K \otimes K_0/ K_0}(K \otimes K_0)^*$, we have $N_{C \otimes K_0}(x) \in N_{K \otimes K_0/ K_0}(K \otimes K_0)^*$. Hence $C \otimes K_0$ is split and $K_0 \subseteq C$ (\cite{FG}, Lemma 5).
\end{proof}
\noindent
One may be tempted to believe that for a distinguished $K$-unitary torus $T$, $H^1(k, T)= 0$. We give below an example to show that this is  false. We also produce an example of a non-distinguished ($k$-anisotropic) torus $T$ such that $H^1(k, T)= 0$.
\\\\
\noindent
{\bf {Example 1:}} Let $k = \mathbb{R}(x)$ and $\delta=-1$. Choose $\alpha \in k^*$ such that $\alpha \notin {k^*}^2$ and $\alpha \neq \delta  ~mod ~{k^*}^2$. Let $K= k(\sqrt{\alpha})$ and $K_0= k(\sqrt{\delta}).$ Then $K_0$ and $K$ are fields. Also note that $K_0= \mathbb{C}(x)$ is a $C_1$ field. Let $L= k \times K_0$, and $T$ be the $K$-unitary torus  associated with the pair $(L, K).$ Then, as in the proof of Theorem \ref{split}, $T\cong R_{K_0/k}^{(1)}(K \otimes K_0)$. Also by Theorem \ref{split}, $ H^{1}(k, T)= {K_0}^*/ N_{K \otimes K_0/K_0}(K \otimes K_0).$ Since $\alpha$ $\neq$ $\delta  ~mod ~{k^*}^2$, $T$ is not distinguished. By Corollary \ref{uni},  $H^1(k, T)=0$  if and only if the binary form $<1, -\alpha>$ becomes universal over $K_0$.
Since $K_0$ is a $C_1$ field, all binary forms over $K_0$ are universal, in particular  $<1, -\alpha>$ becomes universal over $K_0$.  Hence $H^1(k, T)= 0$. Note that since $\alpha \neq \delta  ~mod ~{k^*}^2$, $K \otimes K_0$ is a field. Hence by
(\cite{PR4}, Example on Pg. 54), $T$ is $k$-anisotropic. This also gives an example of a $k$-anisotropic $K$-unitary torus $T$ such that $H^1(k, T)=  0.$\\\\
\noindent
{\bf {Example 2:}} Let $K= k \times k$ and $L$ be a cyclic cubic field extension of $k$. Let $T$ be the $K$-unitary torus associated with the pair $(L, K)$. By definition, $T= {\bf SU}(E, \tau)$, where $(E, \tau)= (L \otimes K, 1 \otimes ~\bar{}~)$. Note that $(E, \tau)= (L \times L, \epsilon)$. Hence $T \cong {\bf L}^{(1)}$. Now $$H^1(k, T) \cong H^1(k, {\bf L}^{(1)}) \cong \frac{k^*}{N_{L/k}(L^*)}.$$ Let $p(X)= X^3-3X-1 \in \mathbb{Q}[X],$ then $p(X)$ is irreducible over $\mathbb{Q}$. Let $L':=\mathbb{Q}[X]/<p(X)>$. Then $L'$ is a cyclic cubic extension of $\mathbb{Q}$ such that $N_{L'/\mathbb{Q}}(L'^*) \neq \mathbb{Q}^*$ (\cite{MT}, Pg. 186). Let $T$ be the $K$-unitary torus associated  with the pair $(L', K)$. Then $T$ is a distinguished torus such that $H^1(\mathbb{Q}, T) \neq 0$.\\\\
\noindent
{\bf {Example 3:}}  Let $T$ be a distinguished $k$-torus arising from a pair $(L, K)$ where $L$ is a cubic \'etale $k$-algebra which is not a field extension and $K$ is a quadratic \'etale $k$-algebra. By Theorem \ref{isom}, $T$ is either $\mathbb{G}_m\times \mathbb{G}_m$ or $R_{K/k}(\mathbb{G}_m)$. In either case, by Hilbert theorem $90$ and Shapiro's Lemma (\cite{KMRT}, Lemma 29.6), $H^1(k, T)= 0$.  \\ \\
\noindent
{\bf {Example 4:}} Let $L= k \times k \times k$ and $K$ be a quadratic \'etale extension of $k$. Then $$SU(E, \tau)= \{ (x, y, z)\in K \times K \times K| x \overline{x}= y \overline{y}= z \overline{z}= 1, xyz=1\} \cong K^{(1)} \times  K^{(1)}. $$ 
Thus ${\bf SU}(E, \tau) \cong {\bf K}^{(1)} \times  {\bf K}^{(1)}$ and $H^1(k, {\bf SU}(E, \tau))= k^*/N_{K/k}(K^*) \times k^*/N_{K/k}(K^*)$. Hence $H^1(k, {\bf SU}(E, \tau))= 0$ if and only if $k^*= N_{K/k}(K^*)$. Let $A$ be an Albert algebra. Let $T$ be the $K$-unitary torus associated with the pair $(L, K)$ such that $H^1(k, T)= 0$.
Let $T \hookrightarrow {\bf Aut}(A)$ over $k$. Since $L$ has trivial discriminant over $k$, by Theorem \ref{KC1}, $K \subset Oct(A)$. Since $H^1(k, T)= 0$, $k^*= N_{K/k}(K^*)$. Hence $f_3(A)=0$.
\begin{theorem} Let $K$ be a quadratic \'etale extension of $k$ and $L= k \times K_0$, where $K_0$ is a quadratic \'etale algebra over $k$. Let $T$ be the $K$-unitary torus associated with the pair $(L, K)$ with $H^1(k, T)=0$. Let $A$ be an Albert algebra over $k$. If there exists an $k$-embedding $T \hookrightarrow {\bf Aut}(A)$, then $K_0 \subset Oct(A)$.\end{theorem}
\begin{proof} If $K_0= k \times k$, then by Example 4, $Oct(A)$ splits. Hence $K_0 \subset Oct(A)$. Let $K_0$ be a field extension. Base changing to $K_0$ we have, 
$$T \otimes K_0= {\bf SU}((L \otimes K_0) \otimes_{K_0} (K_0 \otimes K), \tau \otimes 1) \hookrightarrow {\bf Aut}(A)\otimes K_0.$$ Since $L \otimes K_0$ has trivial discriminant over $K_0$, by Theorem \ref{KC1}, $K_0 \otimes K \subset Oct(A)\otimes K_0$. By doubling, $Oct(A) \otimes K_0= (K \otimes K_0) \oplus (K \otimes K_0).x$, for some $x \in (K \otimes K_0)^{\perp}$, $N_{Oct(A) \otimes K_0}(x) \neq 0$. But since  $H^1(k, T)=0$, by Theorem \ref{split} we have, $K_0^*= N_{K \otimes K_0/ K_0}(K \otimes K_0)^*$. Hence $N_{Oct(A) \otimes K_0}(x) \in N_{K \otimes K_0/ K_0}(K \otimes K_0)^*$. Therefore $Oct(A) \otimes K_0$ is split and by  (\cite{FG}, Lemma 5), $K_0 \subseteq Oct(A).$
\end{proof}
\section{ \'Etale Tits processes of Jordan algebras and applications}
In this section, we develop some results on \'etale Tits processes, in the context of unitary tori. Let $L$ be a cubic \'etale algebra and $K$ be a quadratic \'etale algebra of an arbitrary base field $k$ and $E= L \otimes K$. Let $\tau= 1 \otimes ~\bar{}$, where $\bar{}$ denotes the non-trivial involution on $K$. Suppose $(u, \mu) \in L^* \times K^*$ is such that $N_{L/k}(u)= N_{K/k}(\mu)$. We call the pair $(u, \mu)$ an {\bf admissible pair}. The \'etale Tits process produces an absolutely simple Jordan algebra $J= J(E, \tau, u , \mu)$ of degree $3$ and dimension $9$, with the underlying vector space $L \oplus E$ and with $L= \{(l, 0)|~l \in L\}$ as a subalgebra. Let $(B, \sigma)$ be a central simple algebra over $K$ with an involution $\sigma$ of the second kind and let $(B, \sigma)_+$ contain a cyclic \'etale algebra $L$ over $k$. Then $B= L \otimes K \oplus (L \otimes K)z \oplus (L \otimes K)z^2$ with $z^3= \mu\in K^*$ such that $N_{K/k}(\mu)= 1$. Also the involution $\sigma$ is given by $\sigma(z)= uz^{-1}$ with $u \in L$ such that $N_B(u)= 1$. In this case $(B, \sigma)_+ \cong  J(L \otimes K, \tau, u, \mu)$ (see \cite{KMRT}, Pg. 527 for details). We define \'etale Tits processes $J_1$ and $J_2$ arising from \'etale algebras $L$ and $K$ of dimensions $3, 2$ resp., to be $L$-isomorphic, denoted by $J_1 \cong_L J_2$, if there exists a $k$-isomorphism $J_1 \rightarrow J_2$, which restricts to the subalgebra $L$ of $J_1$ and $J_2$. By (\cite{PR1}, Prop. 3.7) we have the following
\begin{theorem}{\label{admis}}(\cite{PR1}, Prop. 3.7) Let $L, K$ and $E$ be as above. Let $(u, \mu)$ be an admissible pair. For any $w \in E$, $(wu\tau(w), \mu N_{E/K}(w))$ is again an admissible pair and 
$$J(E, \tau, u, \mu) \cong_L J(E, \tau, wu\tau(w), \mu N_{E/K}(w)),$$
via $(a, b) \mapsto (a, bw).$
\end{theorem} 
\begin{remark}{\label{zdivisors}} Note that $J(E, \tau, 1, 1)$ has zero divisors. Choose $x \in E$ such that $\tau(x)= -x$. Then $(0, x)$ is a zero divisor in $J(E, \tau, 1, 1)$, since it is a norm zero element. More generally, as an easy consequence of Theorem \ref{admis}, one can see that if $\mu \in N_{E/K}(E^*)$, then $J(E, \tau, u, \mu)$ has zero divisors. 
\end{remark}
\begin{theorem}{\label{admis1}} For any \'etale Tits construction $J(E, \tau, u, \mu)$, there exists an $L$-isomorphic Tits process $J(E, \tau, u', \mu')$ with $N_{L/k}(u')= 1= \mu'\tau(\mu')$.
\end{theorem}
\begin{proof}
 Take $w= {\mu}^{-1}u$ and apply Theorem \ref{admis}.
\end{proof}
\begin{lemma}{\label{isotope}}
Let  $L, K$ be \'etale  $k$-algebras of dimension $3, 2$ resp. Let $(E, \tau)$ be the $K$-unitary algebra associated with the pair $(L, K)$.
Suppose $\phi:J(E, \tau, u,\mu) \rightarrow J(E, \tau, v ,\nu)$ is an $L$-isomorphism. Then there exists $w \in E$ such that  $u= \phi^{-1}(v) w\tau(w)$ and $\mu= N_{E/K}(w)\nu$ or $\mu= N_{E/K}({w})\overline{\nu}$.
\end{lemma}
\begin{proof} By definition, $(E, \tau)= (L\otimes K, 1 \otimes~\bar{}~)$. By Theorem \ref{admis1}, we may assume $N_{L/k}(v)= \nu \overline{\nu}=1$. 
Let $\phi: J(E, \tau, u,\mu) \rightarrow J(E, \tau, v ,\nu)$  be an $L$-isomorphism. Define $h: L \oplus E \rightarrow L$ and $g:L \oplus E  \rightarrow E$ by 

$$\phi(a, b)= (h(a, b), g(a,b)),$$ for $a \in L, b \in E$. Since $\phi$ is an isomorphism, one can easily check that $g$ and $h$ are $k$-linear maps. Since $\phi$ is an isomorphism of Jordan algebras, it preserves the trace forms on both the algebras. Note that $L^{\perp}$ in $J(E, \tau, u,\mu)$ with respect to the trace form, is the $k$-subspace $\{(0, e)|~e \in E\},$ and similarly for $J(E, \tau, v ,\nu).$ Since $\phi$ restricts to $L,$ $\phi$ maps $L^{\perp}$ in $J(E, \tau, u,\mu)$ to $L^{\perp}$ in $J(E, \tau, v,\nu).$ Hence for $b \in E$, $\phi(0, b)= (0, b')$ for some $b' \in E$. It follows that $h(0, b)= 0$ for all $b \in E$. Therefore $h(a, b)= h(a, 0)$ for all $a \in L,~ b \in E$. We will now on write simply $h(a)$ for $h(a, b)$.  Since $\phi$ is an isomorphism of Jordan algebras, it is easy to check that $h:L \rightarrow L$ is an automorphism. Since $\phi$ restricts to $L$, $\phi(a, 0)= (h(a), 0)$ for all $a \in L$. Hence $g(a, 0)= 0$ for all $a \in L$. It follows that $g(a ,b)= g(0, b)$ for all $a \in L, b \in E$. We will now on write simply $g(b)$ for $g(0, b)$. Again since $\phi$ is an isomorphism of Jordan algebras, it is easy to check that $g: E \rightarrow E$ is a bijection. Since $\phi$ preserves norms, $N(a, b)= N(h(a), g(b))$. Expanding norms we get,
\begin{eqnarray*}
&&N_{L/k}(a) + \mu N_{E/k}(b)+ \overline {\mu N_{E/k}(b)}- t_{L/k}(abu\tau(b)) \\
&= & N_{L/k}(h(a))+ \nu N_{E/k}(g(b))+ \overline {\nu N_{E/k}(g(b))}- t_{L/k}(h(a)g(b)v\tau(g(b))).
\end{eqnarray*}
Putting $a=0$, we get
$$\mu N_{E/k}(b)+ \overline{\mu N_{E/k}(b)}= \nu N_{E/k}(g(b))+ \overline{\nu N_{E/k}(g(b))},~b \in E.$$
\noindent
Since $h$ is an automorphism of $L$, we have $N_{L/k}(a)= N_{L/k}(h(a))$, $a \in L$. Hence we get, $$t_{L/k}(abu\tau(b))= t_{L/k}(h(a)g(b)v\tau(g(b))),~a \in L, ~b \in E.$$
\noindent
Putting $b=1$, we get, $$t_{L/k}(au)=t_{L/k}(h(a)g(1)v\tau(g(1))),~a \in L.$$ Since $g(1)v\tau(g(1)) \in L$, there exist $b \in L$ such that $h(b)= g(1)v\tau(g(1))$. Hence $$t_{L/k}(au)= t_{L/k}(h(a)h(b))= t_{L/k}(h(ab))= t_{L/k}(ab)$$ for all $a \in L$. Since the trace bilinear form $T(a, b):= T_{L/k}(ab)$ on $L/k$ is non-degenerate, we have $u= b$. Let $\hat{h}: E \rightarrow E$ be defined by $\hat{h}= h \otimes 1$. Then $\hat{h}$ is the extension of $h$ to a $K$-automorphism of $E$. In particular, $\hat{h}$ commutes with $\tau$. We have, $$u= \hat{h}^{-1}(g(1))\hat{h}^{-1}(v)\hat{h}^{-1}\tau(g(1))= \hat{h}^{-1}(g(1))h^{-1}(v)\tau(\hat{h}^{-1}(g(1))).$$ Hence $u = w h^{-1}(v) \tau(w)= \phi^{-1}(v)w\tau(w)$, where $w= \hat{h}^{-1}(g(1)) \in E$. This proves the first assertion in the Lemma.

Now we prove the assertion on $\mu$. Let $h^{-1}(v)= v_0 \in L$. Then $N_{L/k}(v_0)= N_{L/k}(v)= 1$.
Let $u_1, u_2 \in E$. Define,
$$<u_1> \cong <u_2>~ \text{over}~E~\text{if}~ \text{and}~ \text{only} ~\text{if}~ u_1= w u_2 \tau(w), ~\text{for}~ \text{some} ~w \in E.$$ Hence , $<u> \cong <v_0>$ over $E$.
\vskip0.5mm
\noindent
We now introduce an equivalence on the set of admissible pairs in $L^* \times K^*$ as follows:\\
\noindent
$(u_1, \mu_1) \sim (u_2, \mu_2)$ if and only if there exists $w \in E$ such that $u_2= w u_1 \tau(w)$ and $\mu_2= N_{E/K}(w) \mu_1$ or $\mu_2= N_{E/K}(w) \overline{\mu_1}$. 
\vskip0.5mm
\noindent
{\bf Claim} $(u, \mu) \sim (v_0, \nu)$.
\vskip0.5mm
\noindent
Since $J(E, \tau, u,\mu) \cong J(E, \tau, v ,\nu)$, we have $J(E, \mu) \cong J(E, \nu)$ over $K$ and by (\cite{PT}, Prop. 4.3), $\mu \in \nu N_{E/K}(E^*)$ or $\mu \in \overline{\nu} N_{E/K}(E^*)$. Let $\mu =  \nu N_{E/K}(w)$ or $\mu =  \overline {\nu} N_{E/K}(w)$ as is the case accordingly, for some $w \in E$. Let $v'= w^{-1}u\tau(w^{-1})$. Then $N_{L/k}(v')=\nu \overline{\nu}= 1$ and $  (v', \nu) \sim (u, \mu)$. Now $<u> \cong <v'>$ and $<u> \cong <v_0>$ over $E$. Hence $<v'> \cong <v_0>$ over $E$.\\
\noindent
Therefore, there exists $w'' \in E$ such that $v_0= w''v' \tau(w'')$. 
Let $\lambda = N_{E/K}(w'')$. Since $N_{L/k}(v_0)= N_{L/k}(v')= 1$, we have $\lambda \overline{\lambda}= 1$.
Hence by (\cite{PT}, Lemma 4.5), there exists $w_1 \in E$ such that $\lambda= N_{E/K}(w_1)$ and $w_1\tau(w_1)=1$.\\
Therefore, $$(v', \nu) \sim (v_0,  \lambda \nu) \sim(v_0, \nu).$$
Thus $(u, \mu) \sim (v_0, \nu)= (u, \mu) \sim (h^{-1}(v), \nu)=  (\phi^{-1}(v), \nu).$
Hence, by the definition of the equivalence,  $\mu= N_{E/K}(w)\nu$ or $\mu= N_{E/K}({w})\overline{\nu}$. This completes the proof.
\end{proof}
\begin{remark}{\label{converseisotope}} As a converse to the above lemma, if there exists $w \in E^*$ such that  $u= \phi^{-1}(v) w\tau(w)$ and $\mu= N_{E/K}(w)\nu$ (or $\mu= N_{E/K}({w})\overline{\nu}$), where $\phi \in Gal(L/k)$, then $J(E, \tau, u,\mu)$ is $L$-isomorphic to $J(E, \tau, v ,\nu)$ (resp.  $J(E, \tau, v , \overline{\nu})$). To see this, suppose $\mu= N_{E/K}(w)\nu$. By Theorem \ref{admis}, $$J(E, \tau, u,\mu)= J(E, \tau, \phi^{-1}(v) w\tau(w), N_{E/K}(w)\nu) \cong_L  J(E, \tau, \phi^{-1}(v), \nu). $$ Extend $\phi$ to an automorphism $\hat{\phi}$ of $E$, defined as $\hat{\phi}= \phi \otimes 1$. Note that $\hat{\phi}$ commutes with $\tau$. Consider the map $\psi: J(E, \tau, \phi^{-1}(v), \nu) \longrightarrow J(E, \tau, v, \nu)$ given by $(a, x) \mapsto (\phi(a), \hat{\phi}(x))$ for $a \in L$, $x \in E$. Clearly $\psi((1, 0))= (1, 0)$. We have,
\begin{eqnarray*}
N(\phi(a), \hat{\phi}(x)) &= & N_{L/k}(\phi(a))+ \mu N_{E/k}(\hat{\phi}(x))+ \overline {\mu N_{E/k}(\hat{\phi}(x))}- t_{L/k}(\phi(a)v\hat{\phi}(x)\tau(\hat{\phi}(x)))\\
& = &N_{L/k}(a)+ \mu N_{E/k}(x)+ \overline {\mu N_{E/k}(x)}- t_{L/k}(\hat{\phi}(a\phi^{-1}(v)x\tau(x)))\\
& = & N_{L/k}(a)+ \mu N_{E/k}(x)+ \overline {\mu N_{E/k}(x)}- t_{L/k}(a\phi^{-1}(v)x\tau(x))= N(a, x).
\end{eqnarray*}
 Hence $\psi$ is a $k$-linear bijection preserving norms and identities. Therefore, by Theroem \ref{jordoisomp}, $\psi$ is an isomorphism of Jordan algebras. Also $\psi$ restricts to $L$. Hence $\psi$ is an $L$-isomorphism.
When $\mu= N_{E/K}({w})\overline{\nu}$, a similar argument completes the proof.
\end{remark}
\begin{corollary}{\label{corrlisotope}} Let  $L, K$ be \'etale  $k$-algebras of dimension $3, 2$ resp. Let $(E, \tau)$ be the $K$-unitary torus associated  with the pair $(L, K)$. There exists an  $L$-isomorphism $J(E, \tau, u,\mu) \cong J(E, \tau, 1, 1)$ if and only if there exists  $w \in E$ such that $u= w\tau(w)$ and $\mu= N_{E/K}(w)$. 
\end{corollary}

\begin{theorem}{\label{isomclas}} There exists a surjective map from $H^1(k, {\bf SU}(E, \tau))$ to the set of $L$-isomorphism classes of \'etale Tits process algebras arising from $(L,K)$.
\end{theorem}
\begin{proof} Let $X$ denote the set of $L$-isomorphism classes of \'etale Tits process algebras arising from $(L,K)$. Given an \'etale Tits process $J$, let $[J]$ denote the $L$-isomorphism class of $J$. Let $\phi: H^1(k, {\bf SU}(E, \tau)) \rightarrow X$ be defined by $\phi([(u, \mu)]):= [J(E, \tau, u, \mu)]$. Let $[(u, \mu)] \in H^1(k, {\bf SU}(E, \tau))$ and $[(u, \mu)]= [(v, \nu)]$. Then $u= vw\tau(w)$ and $\mu= N_{E/K}(w)\nu$ for some $w \in E$. Hence by Theorem \ref{admis}, $J(E, \tau, u, \mu)$ is $L$-isomorphic to $J(E, \tau, v, \nu)$. Therefore $\phi$ is well defined. Clearly $\phi$ is onto.
\end{proof}
\noindent
As an easy consequence of the above theorem, we note that if $H^1(k, {\bf SU}(E, \tau))=0,$ then all \'etale Tits process algebras arising from $(L,K)$ are isomorphic. More precisely, we have,

\begin{theorem}{\label{titsisom}}
 Let $L, K$ be a \'etale  $k$-algebras of dimension $3, 2$ resp. and $(E, \tau)$ be the $K$-unitary algebra and $T$ the $K$-unitary torus associated with the pair $(L, K)$. Then $H^1(k, T)=0$ if and only $J(E, \tau, u, \mu) \cong_L J(E, \tau, 1,1),$ for all admissible pairs $(u, \mu) \in L^* \times K^*.$
\end{theorem}
\begin{proof}
Suppose $J(E, \tau, u, \mu) \cong_L J(E, \tau, 1,1)$ for all admissible pairs $(u, \mu) \in L^* \times K^*.$ Let $S$ be as in Theorems \ref{cohomo1}.
\vskip0.5mm
\noindent
{\bf Claim:} $S= N_{E/L}(E^*)$ and $K^{(1)}= N_{E/K}(U(E, \tau))$.
\vskip0.5mm
\noindent
Let $u \in S$. Since $J(E, \tau, u, \mu) \cong_L J(E, \tau, 1,1)$, by Corollary \ref{corrlisotope}, $u= N_{E/L}(w)= w\tau(w)$ and $\mu= N_{E/K}(w)$ for some $w \in E$. Hence $S= N_{E/L}(E^*)$.  Let $\mu_0 \in K^{(1)}$. 
Since $J(E, \tau, 1, \mu_0) \cong_L J(E, \tau, 1,1)$, by Corollary \ref{corrlisotope}, $\mu_0= N_{E/K}(w)$ where $w\tau(w)=1$. \\
\noindent
Hence $K^{(1)}= N_{E/K}(U(E, \tau)),$ and by Theorem \ref{dcps}, $H^1(k, T)=0$. The converse follows immediately from Theorem \ref{isomclas}.
\end{proof}
\begin{corollary}{\label{deg}} Let $L$ be a cubic \'etale $k$-algebra with discriminant $\delta$ and $K= k(\sqrt{\alpha})$ be a quadratic \'etale $k$-algebra. Let $T$ be the $K$-unitary torus associated with the pair $(L, K)$ and $H^1(k, T)=0$. Let $B$ be any degree $3$ central simple algebra over $k(\sqrt{\alpha\delta})$ with an involution $\sigma$ of the second kind such that $L \subseteq (B, \sigma)_+$. Then $B \cong M_3(k(\sqrt{\alpha\delta}))$ and $\sigma$ is distinguished. 
\end{corollary}
\begin{proof} By (\cite{PT}, Theorem 1.6), there exists an admissible pair $(u, \mu) \in L^* \times K^*$ such that $ \phi: (B, \sigma)_+  \cong J(E, \tau, u , \mu),$ where the isomorphism $\phi$ restricts to the identity of $L$. Since  
 $H^1(k, T)=0$, by Theorem \ref{titsisom},  $(B, \sigma)_+ \cong J(E, \tau, u , \mu) \cong_L J(E, \tau, 1, 1)$.  Since $J(E, \tau, 1, 1)$ is reduced (see Remark \ref{zdivisors}), by (\cite{PR10}, Theorem 1, cf. \cite{PT}, Theorem 1.4), $B \cong M_3(k(\sqrt{\alpha\delta})).$ By Lemma \ref{copys}, there exists $v \in L$ such that $Int(v) \circ \sigma$ is distinguished. Since $\phi$ restricts to the identity of $L$, taking isotopes with respect to $v$ on both sides, we have $(B, Int(v) \circ \sigma)_+ \cong {J(E, \tau, u , \mu)}^{(v)}$ (see \cite{PR1}, Prop. 3.9). By (\cite{PR1}, Prop. 3.9) ${J(E, \tau, u , \mu)}^{(v)} \cong J(E, \tau,uv^{\#}, N(v)\mu) \cong J(E, \tau, 1, 1)$.
 Hence  $(B, \sigma)_+ \cong (B, Int(v) \circ \sigma)_+$. By (\cite{KMRT}, Prop. 37.6), we have $f_3((B, \sigma))= f_3((B, Int(v) \circ \sigma))= 0$ and  $\sigma$ is distinguished. 
\end{proof}
\begin{corollary}{\label{degg}} Let the hypothesis be as in Corollary \ref{deg}. Let $B$ be any degree $3$ central simple algebra over $k(\sqrt{\alpha\delta})$ with an involution of the second kind such that $L \subseteq B$. Then 
$$B \cong M_3(k(\sqrt{\alpha\delta}).$$ 
\end{corollary}
\begin{proof} By (\cite{HKRT}, Prop.17, cf. \cite{KMRT}, Cor. 19.30), there exists an involution $\sigma$ on $B$ such that $L \subseteq (B, \sigma)_+$. Hence by the above corollary $B$ splits.
\end{proof}
\noindent
In view of Corollary \ref{degg}, when $L$ is a cubic \'etale algebra over $k$ with trivial discriminant, we have the following
\begin{corollary}{\label{deggg}} Let $L$ be a cubic \'etale algebra over $k$ with trivial discriminant and $K$ be a quadratic \'etale $k$-algebra. Let $T$ be the $K$-unitary torus  associated with the pair $(L, K)$ and $H^1(k, T)=0$. Let $B$ be any degree $3$ central simple algebra over $K$ with an involution $\sigma$ of the second kind such that $L \subseteq B$. Then $B \cong M_3(K)$.
\end{corollary}
\begin{proof} Let $L$ be as in the hypothesis. When $L= k \times k \times k,$ it is immediate that $B \cong M_3(K).$ When $L$ is a cubic cyclic field extension of $k$, by Corollary \ref{degg}, we get the desired result.
\end{proof}
\subsection{Applications to groups of type $A_2$ and $G_2$}
We now give some applications based on the cohomology computation of tori done in section \S 4.1. We shall consider cohomology of maximal tori in groups of type $A_2$ and $G_2$. These tori arise from six dimensional unitary algebras, hence we can compute their cohomology using Theorems \ref{h1}, \ref{dcps} with $n=3$. We deduce that a group of type $G_2$ splits if and only if it contains a maximal torus whose first cohomology vanishes. A weaker result holds for groups of type $A_2$. We need a variant of (\cite{HKRT}, Prop. 17) for our purpose.

\begin{lemma}{\label{copys}}  Let $F= k(\sqrt{\alpha})$ be a quadratic \'etale $ k$-algebra and $B$ be a degree $3$ central simple algebra over $F$ with an involution $\sigma$ of the second kind. Let $L$ be a cubic \'etale algebra such that $L \subseteq (B, \sigma)_+$. Then there exists $l \in L$ with $N_{L/k}(l) \in {k^*}^2$ such that $Int(l) \circ \sigma$ is distinguished.
\end{lemma}
\begin{proof} Since $L \subseteq (B, \sigma)_+$, by (\cite{HKRT}, Proposition 11), there exists $\mu \in L^*$ with $N_{L/k}(\mu) \in {k^*}^2$
such that $$Q_{\sigma}= <1,2,2 \delta>\perp<2>.<<\alpha \delta>>.~ t_{L/k}(<\mu>).$$ Let $\lambda_0 \in L^*$ be such that $t_{L/k}({\lambda}_0)= 0$ and let $\lambda:= \frac{\lambda_0}{N_{L/k}(\lambda_0)}$. Then $\lambda \in L^*$ and $N_{L/k}(\lambda)\in {k^*}^2$. Hence there exists $\xi \in k^*$ such that $N_{L/k}(\lambda\mu^{-1})= \xi^2$. Consider $\psi:= Int(\xi\lambda^{-1}\mu)\circ \sigma= Int(\lambda^{-1}\mu)\circ \sigma$. 
\vskip0.5mm
\noindent
 {\bf Claim:} $\psi$ is a distinguished involution. 
\vskip0.5mm
\noindent
We will use the proofs of (\cite{HKRT} Prop. 17, Corollary 14). Since $\lambda^{-1}\mu\in L$, we have $L \subseteq (B, \psi)_+$. Let $q: L \rightarrow L$ be defined by, $lq(l)= n_{L/k}(l)$. By (\cite{HKRT}, Proposition 13), we have
$$Q_{\psi}= <1,2,2 \delta>\perp<2>.<<\alpha \delta>>. ~t_{L/k}(<q(\xi\lambda^{-1}\mu)\mu>).$$ It is easy to check that $q(\xi\lambda^{-1}\mu)= \lambda \mu^{-1}$. Hence 
\begin{eqnarray*}
Q_{\psi} &=& <1,2,2 \delta>\perp<2>.<<\alpha \delta>>.~ t_{L/k}(<\lambda>)\\
&= &<1, 1,1> \perp <2 \delta>.<<\alpha>>.~t_{L/k}(<\lambda>).
\end{eqnarray*}
Let $(B, \sigma)_+^{\circ}= \{ x \in (B, \sigma)_+|~ T_B(x)= 0\}$ and ${Q_{\psi}}^{\circ}$ denote the restriction of ${Q_{\psi}}$ to $(B, \sigma)_+^{\circ}$. Then ${Q_{\psi}}^{\circ}= <2>.(<1, 3> \perp <\delta>.<<\alpha>>.~t_{L/k}(<\lambda>)$. Since $t_{L/k}(\lambda)= 0$, the form $t_{L/k}(<\lambda>)$ is isotropicover $k$ and the Witt index of $<<\alpha>>.~ t_{L/k}(<\lambda>)$ is at least two. Hence by (\cite{KMRT}, Theorem 16 (c)), $\psi$ is distinguished.
\end{proof}
\begin{theorem}{\label{cohomoo}} Let $F= k(\sqrt{\alpha})$ be a quadratic \'etale $ k$-algebra and $B$ be a degree $3$ central simple algebra over $F$ with an involution $\sigma$ of the second kind. Let $T$ be a maximal torus of ${\bf SU}(B,\sigma)$. If $H^1(k, T)= 0$ then $\sigma$ is distinguished.
\end{theorem}
\begin{proof} Let $T$ be a maximal torus of ${\bf SU}(B,\sigma)$. By Theorem \ref{imp}, $T \cong {\bf SU}(E, \sigma)$, where $(E, \sigma) \subseteq (B,\sigma)$ is an $F$-unitary algebra. Let $L= E^{\sigma}$ and $Disc(L)= \delta$. By Lemma \ref{count}, $E= L\otimes F$. By Lemma \ref{copys}, there exists $l \in L$, $N_{L/k}(l) \in {k^*}^2$ such that $Int(l) \circ \sigma$ is distinguished. Let $\psi:= Int(l) \circ \sigma$ and $S= \{u \in L^*| ~N_{L/k}(u) \in N_{F/k}(F^*)\}$. Since $H^1(k, T)= 0$, by Theorem \ref{dcps}  $\frac{S}{ N_{E/L}(E^*)}= \{1\}$. Let $u \in S$. Then $u= w\sigma(w)$ for some $w\in E$ and $N_{L/k}(u)= \gamma\overline{\gamma}$  for some $\gamma\in F$. 
Consider the Albert algebra $A:= J(B,\sigma,u,\gamma)$. By (\cite{KMRT} Lemma 39.2), $J(B,\sigma,u,\gamma)\cong J(B,\sigma,w'u\sigma(w'), N_B(w')\gamma)$ for all $w'\in B^*$. Hence for $w'=w^{-1}$, we have $A= J(B,\sigma,u,\gamma) \cong J(B, \sigma, w\sigma(w), \gamma) \cong J(B,\sigma,1,\rho)$, where $\rho= N_B(w)^{-1}\gamma$. Therefore, $f_3(A)= f_3(B, Int(u)\circ \sigma)=  f_3(B, \sigma)$ for all $u \in S$. Taking  $u= l \in S$, we get $f_3(A)= f_3(B,\sigma)= 0$.  Hence $\sigma$ is distinguished.
\end{proof}
\begin{remark} A converse of the above theorem holds when $B$ is split. Let $F= k(\sqrt{\alpha})$ be a quadratic \'etale $ k$-algebra and  $\sigma$ be a distinguished involution on $M_3(F)$.  Let $L= k \times F$. Note that $L \hookrightarrow M_3(F)$ as a $k$-subalgebra (via the embedding $(\gamma, x) \rightarrow diag(\gamma, x, x)$, $\gamma \in k$, $x \in F$). Since  $\sigma$ is distinguished, by (\cite{HKRT}, Cor 18), there exists a $k$-embedding $L \hookrightarrow (M_3(F), \sigma)_+$. Let $T$ be the $F$-unitary torus  associated with the pair $(L, F)$. By Lemma \ref{maxtorusina2}, $T \hookrightarrow G$ over $k$. Then by case {\bf (ii)} of the proof of Theorem \ref{isom}, $T \cong R_{F/k}(\mathbb{G}_m)$. Hence $T \hookrightarrow {\bf SU}(B, \sigma)$ is a maximal $k$-torus with $H^1(k, T)= 0$.
\end{remark}
\begin{corollary} Let $L$ be a cubic \'etale algebra over $k$ of trivial discriminant and $K$ be a quadratic \'etale $k$-algebra. Let $T$ be the  $K$-unitary torus associated with the pair $(L, K)$ and $H^1(k, T)=0$. Let $B$ be any degree $3$ central simple algebra over $K$ with an involution $\sigma$ of the second kind such that $L \subseteq (B, \sigma)_+$. Then $B \cong M_3(K)$ and $\sigma$ is distinguished. 
\end{corollary}
\begin{proof} Since $T \hookrightarrow {\bf SU}(B, \sigma)$ over $k$ (see Lemma \ref{maxtorusina2}) and $H^1(k, T)=0$, by Theorem \ref{cohomoo}, $\sigma$ is distinguished. Also, by Corollary \ref{deggg}, $B \cong M_3(K)$.
\end{proof}
\begin{theorem}{\label{introg2}} Let $G$ be a group of type $G_2$ over $k$. Then $G$ splits over $k$ if and only if there exist a maximal $k$-torus  $T\subset G$ such that $H^1(k,T)=0$. 
\end{theorem}
\begin{proof} 
 Let $T\subset G$ be a maximal $k$-torus such that $H^1(k,T)=0$. As in $\S 2.8$, there exists a quadratic \'etale $k$-algebra $K$ and $h \in GL_3(k)$ such that $T \subseteq  {\bf SU}(M_3(K), *_h)  \subseteq G$, where $*_h$ denotes the involution on $M_3(K)$ given by $*_h(X)=h^{-1}\overline{X}^th$. Since $H^1(k,T)=0$, $*_h$ is a distinguished involution (see Theorem \ref{cohomoo}). Hence by (\cite{NH}, Theorem 4.4), $G$ splits over $k$. For the converse, we choose $T$ to be a split maximal $k$-torus in $G$, then $H^1(k,T)=0$.
\end{proof}
\noindent
{\bf The Real Case}
\vskip0.5mm
\noindent
Let $G$ be a group of type $F_4$ over $\mathbb{R}$. Let $L, K$ be \'etale algebras over $\mathbb{R}$ of dimension $3, 2$ resp. and $T$ be the $K$-unitary torus associated with the pair $(L, K)$. Suppose $T \hookrightarrow G$ over $\mathbb{R}$. If $H^1(\mathbb{R}, T)=0$ then $f_5(G)= 0$. Note that $K= \mathbb{R} \times \mathbb{R}$ or $\mathbb{C}$. If $K= \mathbb{R} \times \mathbb{R}$, then by Theorem \ref{albi1}, $f_5(A)= 0$. Suppose $K= \mathbb{C}$. Note that $L= \mathbb{R} \times \mathbb{R} \times \mathbb{R}$ or $L= \mathbb{R} \times \mathbb{C}$. If $L= \mathbb{R} \times \mathbb{R} \times \mathbb{R}$, then by case {\bf (i)} of proof of Theorem \ref{isom}, $T \cong \mathbb{G}_m \times \mathbb{G}_m$ over $\mathbb{R}$. Hence $\mathbb{R}$-rank of $G \geq 2$ and by Lemma \ref{isotropycond}, $f_3(G)= f_5(G)= 0$. Suppose $L= \mathbb{R} \times \mathbb{C}$. Then by case {\bf (ii)} of proof of Theorem \ref{isom}, $T \cong R_{\mathbb{C}/\mathbb{R}}(\mathbb{G}_m$) over $\mathbb{R}$. Hence $\mathbb{R}$-rank of $G \geq 1$ and by Lemma \ref{isotropycond}, $f_5(G)= 0$.
\begin{remark}
The real case along with Example 4, leads us to raise the following question: Let $L, K$ be \'etale algebras of dimension $3, 2$ resp. and $T$ be the $K$-unitary torus associated to the pair $(L, K)$. Let $G$ be a group of type $F_4$ defined over $k$ and $T\hookrightarrow G$ over $k$, then does $H^1(k,T)=0$ imply $f_5(G)=0$? Though we have not been able to settle this over an arbitrary field, we can prove a weaker result. 
\end{remark}
\noindent
We first state a result which will be required,
\begin{theorem} (Knebusch norm principle) (\cite{TYL}, Chap. VII, Thm. 5.1){\label{knp}} Let $K/F$ be a finite extension of degree $n$ and $q$ be a quadratic form over $F$. Let $x \in K^*$. If $x \in D_K(q_K)$ then $N_{K/F}(x)$ is a product of $n$ elements of $D_F(q)$. (In particular $N_{K/F}(x) \in [D_F(q)]$). Hence if $q$ is a Pfister form ove $F$ and $q_K$ is isotropic, then $N_{K/k}(K^*) \subseteq D_F(q)$.
\end{theorem}
\begin{theorem}{\label{introf4}} Let $L, K$ be \'etale algebras over $k$ of dimension $3, 2$ resp. and $E$ be the $K$-unitary algebra and $T$ the $K$-unitary torus associated with the pair $(L, K)$. Let $G$ be a group of type $F_4$ (resp. $G_2$ or a simple simply connected group of type $A_2$). Assume there is a $k$-embedding $T \hookrightarrow G$. If $H^1({\bf U}(E, \tau))= 0$ then $f_5(A)= 0$ (resp. Oct(G) splits).
\end{theorem}
 \begin{proof} Consider the exact sequence \\
$$ 1 \longrightarrow {\bf U}(E, \tau) \longrightarrow E^* \xrightarrow{N_{E/L}} L^* \longrightarrow 1.$$
The long exact cohomology sequence yields the exact sequence,
$$ {\bf U}(E, \tau) \longrightarrow E^* \xrightarrow{N_{E/L}} L^* \longrightarrow H^1({\bf U}(E, \tau)) \longrightarrow 1.$$
Hence $$H^1({\bf U}(E, \tau))= L^*/ N_{E/L}(E^*).$$
Since $H^1({\bf U}(E, \tau))= 0$, $L^*= N_{E/L}(E^*).$ Let $K= k(\sqrt {\alpha})$ and $q=<1, -\alpha>$. If $K= k \times k$ then by Theorem \ref{albi1}, $f_5(G)= 0$. Hence we may assume that $K$ is a field extension. If $L= k \times K_0$, for some quadratic \'etale algebra $K_0$ over $k$, then $H^1({\bf U}(E, \tau))= k^*/ N_{K/k}(K^*) \times  K_0^*/ N_{K_0 \otimes K/K_0}({K_0 \otimes K})^* .$ Since $H^1({\bf U}(E, \tau))= 0$, $k^*=  N_{K/k}(K^*)$. Hence $q$ is universal over $k$. By Theorem \ref{albi1}, $q$ divides $f_5(G)$ and is a subform of $f_5(G)$ (see Theorem \ref{subformfactor}). Hence  we have $f_5(G)= 0$. Suppose $L$ is a field extension. Now $q_L= N_{E/L}(E^*)$. Since $q_L$ splits over $E$, by Theorem \ref{knp}$, N_{E/L}(E^*) =L^* \subseteq D_L(q_L)$. Hence $q$ is universal over an odd degree extension $L$ of $k$. By Springer's theorem, $q$ is universal over $k$. Let $G$ be a group of type $F_4$. By Theorem \ref{albi1}, $q$ divides $f_5(G)$, hence $f_5(G)= 0$. Let $G$ be a group of type $G_2$ or $A_2$. By Theorem \ref{KC},  $q$ divides $f_3(G)$, hence $f_3(G)= 0$. Thus $Oct(G)$ splits.
\end{proof}
\section{Generation of groups $A_2$, $G_2$, $D_4$ and $F_4$ by rank-$2$ tori}
In this section we deduce the number of rank-$2$-tori required for the generation of groups of type $A_2$, $G_2$, and $F_4$ arising from division algebras and subgroups of type $D_4$ of ${\bf Aut}(A)$ for $A$ an Albert division algebra. Let $G$ be a simple, simply connected group of type $A_2$ over $k$. Then the minimum number of maximal tori required to generate $G$ is $2$, when $k$ is a perfect field. Let $H_i$, $i= 1, 2$, be algebraic subgroups of an algebraic group $G$. By $<H_1, H_2>$ we will denote the algebraic subgroup of $G$ generated by $H_i$, $i= 1, 2.$ We will often use the Borel-De Siebenthal algorithm. For details see (\cite{BS}). Let $X$ and $Y$ be types of root systems. If $X$ is a subsystem of $Y$, we write $X \subseteq Y$.
We begin with,
\begin{lemma}{\label{repeat}} Let $G$ be a $k$-anisotropic, connected, reductive algebraic group over a perfect field $k$. Let $H$ be a connected subgroup of $G$. Then $H$ is a reductive, $k$-anisotropic subgroup.
\end{lemma}
\begin{proof} Since $G$ is a $k$-anisotropic, by Prop. \ref{unipotred}, $G(k)$ has no non-trivial unipotents. Hence $H(k)$ has no non-trivial unipotents and $R_u(H)(k)= \{1\}$. Since $k$ is perfect, by density of $k$-points it follows that, $R_u(H)= \{1\}$.
\end{proof}
\begin{theorem}{\label{fact}}  Let $k$ be a perfect infinite field and $F$ be a quadratic \'etale $ k$-algebra. Let $(B, \sigma)$ be a degree $3$ central division algebra over $F$ with an involution $\sigma$ of the second kind. Let $G= {\bf SU}(B,\sigma)$.
Let $E_1,E_2 \subset B$ be $F$-unitary subalgebras of $B$ such that $\sigma$ restricts to $E_1$ and $E_2$. Let $\sigma_i= \sigma|_{E_i}$. Assume that ${\bf SU}(E_1, {\sigma_1})\neq {\bf SU}(E_2, {\sigma_2})$.  Then $$G= <{\bf SU}(E_1, {\sigma_1}), {\bf SU}(E_2, {\sigma_2})>.$$
\end{theorem}
\begin{proof} Let $H= <{\bf SU}(E_1, {\sigma_1}), {\bf SU}(E_2, {\sigma_2})>$. Then $H$ is a connected $k$-subgroup of $G$. Since $B$ is a division algebra, $G$ is a $k$-anisotropic group (see Theorem \ref{involution}). Notice that since ${\bf SU}(E_i, {\sigma_i}),~i=1,2$ are maximal tori of $G$, $H$ is a non-toral subgroup. By Lemma \ref{repeat}, $H$ is a connected, reductive, $k$-anisotropic, non-toral subgroup of $G$. Since $G$ has absolute rank-$2$, $[H, H]$ is a semisimple group of absolute rank $1$ or $2$. Hence $[H, H]$ must be of type $A_2,~A_1, ~A_1\times A_1,~ G_2~\text{or}~B_2= C_2$. By the Borel-De Siebenthal algorithm, $A_1 \times A_1, B_2 \nsubseteq A_2$. Notice that $G_2 \nsubseteq A_2$ (since Lie algebra of $G_2$ has dimension $14$ whereas the dimension of Lie algebra of $A_2$ is $8$). If $[H,H]$ is of type $A_1$, then $G$ has a $k$-torus $S$ of absolute rank $1$, $S \subseteq [H,H]$. Necessarily, $S= {\bf M}^{(1)}$, the norm torus of a quadratic extension $M/k$ (\cite{Vos}, Chap.II, \S IV, Example 6). But then, $S$ splits over $M$ and hence $G$ becomes isotropic over $M$. 
By Prop. \ref{coprime3}, $B$ remains a division algebra over $M$. Hence by Theorem \ref{involution}, $G$ remains anisotropic over $M$, a contradiction. Therefore $[H,H]$ cannot have type $A_1$. Hence $[H, H]$ must be of type $A_2$. Now $H \subseteq G= [G, G]= [H, H] \subseteq H$. Therefore $H= G$.
\end{proof}
\noindent
Let $G$ be a group of type $G_2$ (resp. $F_4$) defined over $k$. We now calculate the number of rank-$2$ tori required to generate $G$. In (\cite{NH}, Theorem 3.11, 4.1) we proved that a group of type $G_2$ is generated by its $k$-subgroups of type $A_1$ and a group of type $F_4$ is generated by its $k$-subgroups of type $A_2$. The results below are continuation of that. We first prove that a group of type $G_2$ (resp. $F_4$) is also generated by two $k$-subgroups of type $A_2$ (resp. $D_4$). Using this we deduce that a group of type $G_2$ (resp. $F_4$) is generated by three (resp. four) rank-$2$ tori.
\begin{theorem}{\label{g2gen}} Let $C$ be an octonion division algebra over $k$, where $k$ is a perfect (infinite) field. Then $G= {\bf Aut}(C)$ is generated by two $k$-subgroups of type $A_2$.
\end{theorem}
\noindent
For the proof of this theorem, we need the following
\begin{proposition}{\label{toralnh3}} (\cite{NH}, Prop. 4.1) Let $C$ be an octonion division algebra over $k$. Let $G= {\bf Aut}(C)$. Let $H$ be a proper connected reductive, non-toral subgroup of $G$ defined over $k$. Then $[H,H]$ is of type $A_1$, $A_1 \times A_1$ or $A_2$.
\end{proposition}
\noindent
We now prove Theorem \ref{g2gen}.
\begin{proof} Choose quadratic subfields $K_1,K_2 \subset C$ such that $K_1 \cap K_2= k$. Let $H_i= {\bf Aut}(C/K_i)$, $i=1,2$. By Theorem \ref{typeA2}, $H_i$, $i=1, 2$, are simple, simply connected subgroups of type $A_2$. Let $H$ denote the closed subgroup of $G$ generated by $H_i$, $i=1,2$. By Lemma \ref{repeat}, $H$ is a connected, reductive, $k$-anisotropic, non-toral subgroup of $G$ containing $H_i$, $i=1,2$ properly. By Prop. \ref{toralnh3}, $[H, H]$ is of type $A_1$, $A_1 \times A_1$, $A_2$ or $G_2$. Now $H_1= [H_1, H_1]\subseteq [H, H]$. Since $H_1$ is of type $A_2$, by the Borel-De Siebenthal algorithm, $[H, H]$ cannot be of type $A_1$ or $A_1\times A_1$. Therefore $[H,H]$ must be of type $A_2$ or $G_2$. If $[H,H]$ is of type $A_2$ then $[H,H]= H_i$, $i=1,2$. Hence ${\bf Aut}(C/K_1)= {\bf Aut}(C/K_2)$ and hence ${\bf Aut}(C/K_1)= {\bf Aut}(C/Q)$ where $Q$ denotes the quaternion subalgebra of $C$ generated by $K_1$ and $K_2$. This is a contradiction since ${\bf Aut}(C/Q)$ is of type $A_1$ (Theorem \ref{typeA2}) while ${\bf Aut}(C/K_1)$ is of type $A_2.$ Hence $[H,H]$ is of type $G_2$. Now $H \subseteq G= [G, G]= [H, H] \subseteq H$. Therefore $H= G$.
\end{proof}
\noindent
From Theorems \ref{fact}, \ref{g2gen}, we can immediately deduce that when $C$ is an octonion division algebra over a perfect (infinite) field $k$, then $G= {\bf Aut}(C)$ is generated by four $k$-tori of rank-$2$. However, we can do better:
\begin{theorem}{\label{octal}} Let $C$ be an octonion division algebra over $k$, where $k$ is a perfect (infinite) field. Then $G= {\bf Aut}(C)$ is generated by three $k$-tori of rank-$2$.
\end{theorem}
\begin{proof}
The algebraic group  $G={\bf Aut}(C)$ is a connected, simple algebraic group of type $G_2$, in particular, $G$ has absolute rank-$2$. By Theorem \ref{g2gen}, $G$ is generated by two subgroups  $H_i,~i=1,2$, of type $A_2$ with $H_1 \neq H_2$. Choose a maximal $k$-torus $T \subseteq H_1$ such that $T \nsubseteq H_2$. Let $H= <T, H_2>$ be the (closed) subgroup generated by $T$ and $H_2$. Since $C$ is a division algebra, $G$ is $k$-anisotropic (Prop. \ref{red1}). By Lemma \ref{repeat}, $H$ is a connected reductive $k$-anisotropic non-toral subgroup of $G$ containing $H_2$ properly. Using same arguments as in Theorem \ref{g2gen}, it follows that $[H,H]$ must be of type $A_2$ or $G_2$. If $[H,H]$ is of type $A_2$, then $[H,H]= H_2$ (since $H_2= [H_2, H_2] \subseteq [H, H]$ and both are of type $A_2$). Now $H= [H, H].{Z(H)}^o= H_2.{Z(H)}^o$ and $Z(H)= \cap{T_i}$, $T_i$'s are maximal tori  of $H$ (\cite{BA1}, \S 13.17, Cor. 2). Since any maximal torus in $H_2$ is maximal in $H$ we have, $Z(H)\subset H_2$. Hence $H= H_2$ and $T\subseteq H_2$, contradicting the choice of $T$. Therefore $[H,H]$ is of type $G_2$. Now $H \subseteq G= [G, G]= [H, H] \subseteq H$. Hence $H= G$. The result now follows since $H_2$ is itself generated by two rank-$2$ $k$-tori.
\end{proof}
\noindent
One can derive similar results for  groups of type $D_4 \subseteq F_4$ and $F_4$. We now find the number of rank-$2$ tori required to generate  $D_4 \subseteq F_4$ and $F_4$ type groups arising from division algebras.

\begin{theorem}{\label{xx}} Let $A$ be an Albert division algebra over a perfect (infinite) field $k$ and $G= {\bf Aut}(A)$. Let $H= {\bf Aut}(A/L)$ where $L$ is a $3$-dimensional subalgebra of $A$. Then $H$ is generated by three rank-2 tori over $k$.
\end{theorem}
\noindent
For the proof of this theorem we need the following,
\begin{lemma}{\label{ntoralnh3}}Let $A$ be an Albert division algebra over a field $k$. Let  Let $H$ be a subgroup of $G$ of type $D_4$ and $H_0 \subseteq H$ be a non-toral reductive $k$-subgroup. Then $[H_0,H_0]$ is of type $A_2$ or $D_4$.
\end{lemma}

\begin{proof} By (\cite{NH}, Theorem 3.10) $[H_0, H_0]$ is of type $A_2$, $A_2 \times A_2$ or $D_4$. By the Borel-De Siebenthal algorithm, $A_2 \times A_2 \nsubseteq D_4$ and hence $[H_0, H_0]$ must be of type $A_2$ or $D_4$.
\end{proof}
\noindent
We now prove Theorem \ref{xx}.
\begin{proof} By (\cite{NH} Theorem 3.9), $H= <H_1, H_2>$, where $H_i= {\bf Aut}(A/S_i)$ where $S_i$ are $9$-dimensional subalgebras of $A$ with $S_1 \cap S_2= L$.  Note that $H_1 \cap H_2= \{1\}$. By Theorem \ref{typeA}, $H_i$, $i=1, 2$, is simple, simply connected subgroup of type $A_2$. Also $H_i$, $i=1, 2,$ arise from division algebras.  \vskip0.5mm
\noindent
{\bf Claim:} We can choose a maximal torus $S\subseteq H_1$ such that $S \nsubseteq {\bf Aut}(A, S_2).$
\vskip0.5mm
\noindent
 If not, then $H_1 \subseteq {\bf Aut}(A, S_2)$ (since $H_1$ is generated by its maximal $k$-tori). Note that $H_2= {\bf Aut} (A/S_2) \subseteq {\bf Aut}(A, S_2)$. Hence $H \subseteq {\bf Aut}(A, S_2)$, a contradiction, since $D_4 \nsubseteq A_2 \times A_2$. Thus we can choose a maximal $k$-torus $S \subseteq H_1$ such that $S  \nsubseteq {\bf Aut}(A, S_2)$. Let $H_0:= <S, H_2> \subseteq H$. Then, by Theorem \ref{fact}, $H_0$ is generated by three rank $2$ $k$-tori. We will prove that $H_0= H$. By Lemma \ref{repeat}, $H_0$ is a connected reductive, $k$-anisotropic, non-toral subgroup of $G$ containing $S$ and $H_2$ properly. By Lemma \ref{ntoralnh3},
$[H_0, H_0]$ is of type $A_2$ or $D_4$. If $[H_0, H_0]$ is of type $A_2$, then $ H_2= [H_2, H_2]= [H_0, H_0]$ (since $H_2$ is of type $A_2$). This shows that $H_2$ is a normal subgroup of $H_0$. Also $S \cap H_2= \{1\}$, hence $H_0= <S, H_2>= S.H_2$. Now $$H_0=  [H_0, H_0]. Z(H_0)^o= H_2.Z(H_0)^o.$$ Consider the projection maps $\tau$ and $\tau'$ given by,
$$Z(H_0)^o\subseteq H_0= S. H_2 \xrightarrow{\tau} S, H_0= S. H_2 \xrightarrow{\tau'} H_2.$$
Since $H_0 \neq H_2$, we have $Z(H_0)^o \neq \{1\}$. Since $A$ is a division algebra, ${\bf Aut}(A)$ does not have rank-$1$ $k$-tori (Theorem \ref{involution}). Hence $Z(H_0)^o$ is a rank-$2$ $k$-torus. Since $\tau(Z(H_0)^o)$ is connected, $\tau(Z(H_0)^o)= S$ or $\{1\}$. If $\tau(Z(H_0)^o)= \{1\}$, then $Z(H_0)^o \subseteq H_2$, hence $H_0= H_2$, a contradiction, since $S \cap H_2 =\{1\}$. Therefore $\tau(Z(H_0)^o)= S$.\\
%and $Z(H_0)^o= \{sh | ~s \in S ~\text{is}~\text{arbitary}, \text{for}~ \text{some} ~h \in H\}$.\\
\noindent
Let $H'= \tau'({Z(H_0)}^o)$. Note that $1 \in H'$.
%= \{ h \in H_2|~ sh \in Z(H_0)^o \text{for}~\text{some} ~s \in S\}$. 
\noindent
 If $H'= \{1\}$ then $Z(H_0)^o= S$. Since $ Z(H_0)^o$ centralizes $H_2$, we see that $Z(H_0)^o$ stabilizes $A^{H_2}$. Therefore $S \subseteq {\bf Aut}(A, S_2)$, a contradiction. Hence $H' \neq \{1\}$. 
\vskip0.5mm
\noindent
{\bf  Claim:} $H'$ is a rank-$2$ $k$-torus of $H_2$. 
\vskip0.5mm
\noindent
%$H'= \tau'(Z(H_0)^o)$, 
We have, for $s_i h_i \in  Z(H_0)^o,~(s_1h_1)(s_2h_2)= s_2(s_1h_1)h_2= (s_1s_2)(h_1h_2)$.  
Hence $\tau'$ is a homomorphism. It follows that $H'= \tau'(Z(H_0)^o)$ is a $k$-torus. Now since $H_2$ does not have any rank-$1$ $k$-tori (Theorem \ref{involution}) and $H' \neq \{1\}$, $H'$ is a rank-$2$ $k$-torus of $H_2$.
\vskip0.5mm
\noindent
{\bf Claim:} $S$ centralizes $H'$. 
\vskip0.5mm
\noindent
Let $s \in S$ and $h \in H'$. Since $h \in H'$, there exists $s_0 \in S$ such that $s_0h \in  Z(H_0)^o$.  Since $s_0h \in  Z(H_0)^o$, we have, $$shs^{-1}= ss_0^{-1}s_0hs^{-1}= s_0hss_0^{-1}s^{-1}= s_0hs_0^{-1}= h.$$
Hence $S$ centralizes $H'$ and therefore $S$ stabilizes $A^{H'}$. Since $A^{H'}=S_2$, we have $S \subseteq {\bf Aut}(A, S_2)$,
 a contradiction. Hence $[H_0, H_0]$ cannot be of type $A_2$. Therefore $[H_0, H_0]$ is of type $D_4$. Now $H_0 \subseteq H= [H, H]= [H_0, H_0] \subseteq H_0$. Therefore $H= H_0$ and $H$ is generated by three rank-2 tori over $k$. 
\end{proof}
\begin{theorem} Let $A$ be an Albert division algebra over a perfect (infinite) field $k$. Then $G= {\bf Aut}(A)$ is generated by four rank-2 tori over $k$.
\end{theorem}
\noindent
We first prove the following lemma,
\begin{lemma} Let $A$ be an Albert division algebra over a perfect (infinite) field $k$ and $G= {\bf Aut}(A)$.  Let $H_i:= {\bf Aut}(A/L_i) \subseteq G$, $i= 1, 2$, where $L_1 \neq L_2$ are $3$-dimensional subalgebras of $A$. Then $G$ is generated by $H_i$, $i= 1, 2$. 
\end{lemma}
\begin{proof} Let $H=<H_1, H_2>$. By Lemma \ref{repeat}, $H$ is a connected, reductive, $k$-anisotropic, non-toral subgroup of $G$. By (\cite{NH}, Theorem 3.10), $[H, H]$ is of type $A_2$, $A_2 \times A_2$, $D_4$ or $F_4$. Since $D_4 \nsubseteq A_2, A_2 \times A_2$,  $[H, H]$ is of type  $D_4$ or $F_4$. If $[H, H]$ is of type  $D_4$, then $H_i= [H_i,H_i] \subseteq [H, H]$,  $i= 1, 2$ and $H_i$, $i= 1, 2$, is of type $D_4$, hence $H_i= [H, H]$, $i= 1, 2$, a contradiction since $H_1 \neq H_2$ . Therefore $[H, H]$ is of type $F_4$. Hence $H= G$.
\end{proof}
\noindent
We now give the proof of the above theorem
\begin{proof} By the above lemma, $G= <H_1, H_2>$, $H_i= {\bf Aut}(A/L_i)$ where $L_i,~i=1,2$, are three dimensional subalgebras. Choose a rank-$2$-torus $T \subseteq H_1$ such that $T \nsubseteq H_2$ (otherwise $H_1= H_2$ since $H_i$'s are generated by their rank-$2$ tori). Let $H= <T, H_2>$. By Lemma \ref{repeat}, $H$ is a connected, reductive, $k$-anisotropic, non-toral subgroup of $G$. By (\cite{NH}, Theorem 3.10) the possible types of $[H, H]$ are $A_2$, $A_2 \times A_2$, $D_4$ or $F_4$. Now $H_2= [H_2, H_2] \subseteq [H, H]$.  Since $H$ contains $H_2$ properly, $[H, H]$ cannot be of type  $A_2$ or $A_2 \times A_2$. Suppose $[H, H]$ is of type  $D_4$. Then $H_2 = [H, H]$. Now $H= [H, H].Z(H)^o= H_2.Z(H)^o$. Since the rank of maximal tori of $H$ and $H_2$ is four we have, $Z(H)^o= \{1\}$. Hence $H= H_2$, a contradiction. Therefore $[H, H]$ is of type $F_4$ and $H= G$. 
\end{proof}
\section{\bf Acknowledgement}
I thank the Council of Scientific and Industrial Research, Govt. of India, for its financial support. This research is part of my Ph.D. work and was supported by 
the C.S.I.R. fellowship. I thank Maneesh Thakur for suggesting the problem and his constant guidance throughout the project. I express my thanks to Shripad M. Garge and Dipendra Prasad for many fruitful discussions.


\begin{thebibliography}{99}
\bibitem{A} J. Kr. Arason, \emph{Cohomologische Invarianten quadratischer Formen}, J. Algebra, {\bf 36}, 448-491, 1975.
\bibitem {BA1} A. Borel, \emph{ Linear algebraic groups}, Second edition., Graduate Texts in Mathematics, 126. Springer-Verlag, New York, 1991.
\bibitem{BS} A. Borel and De Siebenthal, \emph{Les sous-groupes ferm\'es connexes de rang maximum des groupes de Lie clos}, Comment. Math. Helv. {\bf 23}, 200-221. 1949-1950.
\bibitem{CPT} C. Beli, P. Gille, T.-Y. Lee, \emph{ On maximal tori of algebraic groups of type $G_2$}, preprint arXiv:1411.6808.
\bibitem{FG} J. C. Ferrar, \emph{Generic Splitting fields of composition algebras}, Trans. Amer. Math. Soc. {\bf 128}, 506-514, 1967.
\bibitem{HKRT}  D. Haile, M. A. Knus, M. Rost, J. P. Tignol, \emph{Algebras of odd degree with involution, trace forms and dihedral extensions}, Israel Journal of Math. {\bf 96}, 299-340, 1996.
\bibitem{HH} H. Hijikata, \emph{A remark on the groups of type $G_2$ and $F_4$}, J. Math. Soc. Japan, {\bf 15}, 159-164, 1963.
\bibitem{NH} Neha Hooda, \emph{Invariants Mod-$2$ and subgroups of $G_2$ and $F_4$}, Journal of algebra {\bf 411} (2014), 312-336.
\bibitem{NJ1} N. Jacobson, \emph{Composition algebras and their automorphisms}, Rend. Circ. Mat. Palermo (2) {\bf 7}, 55-80, 1958.
\bibitem{NJ2} N. Jacobson, \emph{Basic algebra. II.}, Second edition., W. H. Freeman and Company, New York, 1989. 
\bibitem{J} N. Jacobson, \emph{Structure and representations of Jordan algebras}, AMS Colloquium publications, Volume {\bf 39}. AMS Providence, RI, 1968.
\bibitem{NJ4} N. Jacobson, \emph{Some groups of transformations defined by Jordan algebras. II. Groups of type $F_4$}, J. reine angew. Math. {\bf 204}, 74-98, 1960.
\bibitem{KMRT} M. A. Knus, A. Merkurjev, M. Rost, J. P. Tignol, \emph{The Book of Involutions}, AMS. Colloquium Publications, Vol. 44, 1998.
\bibitem{HUL} H. Lausch, \emph{Automorphisms of Cayley algebras over finite fields}, Results in Mathematics {\bf 15} (1989), 343-350.
\bibitem{TYL} T.Y. Lam, \emph{Introduction to quadratic forms over fields}, Graduate Studies in Math. Volume 67, AMS. Providence, Rhode Island, 2004. 
\bibitem{Alf} Alf Neumann, \emph{Automorphismengruppen von Cayleyalgebren}, Diplomarbeit, Wurzburg, 1988.
\bibitem{PT} H.P. Petersson, Maneesh Thakur, \emph{The \'etale Tits process of Jordan algebras revisited}, Journal of Algebra, 88-107.
\bibitem{PR1} H.P. Petersson, M. L. Racine, \emph{Jordan algebras of degree 3 and Tits process, J. Algebra} {\bf 98} (1986), 211-243
\bibitem{PR2} H. P. Petersson, M. Racine, \emph{Albert algebras}, Proceedings of a conference on Jordan algebras (W. Kaup and K. McCrimmon, eds), Oberwolfach 1992, de Gruyter, Berlin, 1994.
\bibitem{PR7} H. P. Petersson, M. Racine, \emph{Enumeration and classification of Albert algebras: reduced models and the invariants mod 2}, Non-associative algebra and its applications (Oviedo, 1993), 334–340, Math. Appl., {\bf 303}, Kluwer Acad. Publ., Dordrecht, 1994. 
\bibitem{PR}  H. P. Petersson, M. Racine, \emph{ Reduced models of Albert algebras}, Math. Z. {\bf 223}, 367-385, 1996.
\bibitem{PR10}     H.P. Petersson, M.L. Racine, \emph{The toral Tits process of Jordan algebras}, Abh. Math.
Sem. Univ. Hamburg {\bf 54} (1984) 251–256.
\bibitem{PR4} V. Platonov, A. Rapinchuk, \emph{Algebraic groups and number theory}, Pure and Applied Math. , Vol {\bf 139}, Academic Press Inc., Boston, MA, 1994, Translated from the 1991 Russian original by R. Rowen. 
\bibitem{PRT1} R. Parimala, R. Sridharan, Maneesh L. Thakur, \emph{A classification theorem for Albert algebras}, Trans. AMS. {\bf  350}(3), 1277-1284, 1998.
\bibitem{R} J.D. Rogawski, \emph{Automorphic representations of unitary groups in three variables}, Annals of Math. Studies 123, Princeton University Press, 1990.
\bibitem{R} R. W.  Richardson, \emph{Conjugacy classes in Lie algebras and algebraic groups}, Annals of Math. Second series, Vol.{\bf 86}, No.1, 1-15, 1967.
\bibitem{Spr} T. A. Springer, \emph{Linear Algebraic Groups}, Second Edition, Progress in Mathematics, Birkhauser, Boston, 1998.
\bibitem{SV} T. A. Springer and F. D. Veldkamp, \emph{Octonions, Jordan Algebras and Exceptional Groups}, Springer Monographs in Mathematics, Springer-Verlag, Berlin, 2000.
\bibitem{MT} A. Singh and M. Thakur, \emph{Reality properties of conjugacy classes in $G_2$}, Israel J. Math. {\bf 145} (2005), 157–192.
\bibitem{AS1} A. Singh, \emph{Reality properties of conjugacy classes in algebraic groups}, Ph.D. thesis, Indian Statistical Institute (2006).
\bibitem{T3} J. Tits, \emph{Classification of algebraic semisimple groups}, Algebraic Groups and Discontinuous Subgroups (Proc. Sympos. Pure Math., Boulder, Colo., 1965) (A. Borel and G. D. Mostow, eds.), Vol.{\bf 9}, AMS., Providence, R.I., 33-62, 1966.
\bibitem{Vos} V. E. Voskresenskii, \emph{Algebraic Groups and Their Birational Invariants}, Translation of Math. Monographs, Vol {\bf 179}, AMS, Providence, RI, 1998.
\bibitem{Wob} M. J Wonenburger, \emph{Automorphisms of Cayley Algebras}, J. Algebra {\bf 12} (1969), 44- 452.

\end{thebibliography}
\end{document}